\documentclass[final,leqno]{siamltex}

\usepackage{amsmath,amssymb}
\usepackage{graphicx}
\usepackage{footnote}

\title{The Geometry of r-Adaptive Meshes Generated Using Optimal Transport Methods}

\author{C.J. Budd\thanks{University of Bath, UK, BA2 7AY (mascjb@bath.ac.uk).},\and
 R. D. Russell\thanks{Simon Fraser University, Burnaby, BC, Canada, V5N IS6 (rdr@cs.sfu.ca).}\and E. Walsh\thanks{Simon Fraser University, Burnaby, BC, Canada, V5N IS6 (ewalsh@sfu.ca).}}

\begin{document}

\maketitle

\begin{abstract}
The principles of mesh equidistribution and alignment
play a fundamental role in the design of adaptive methods \cite{HRBook},
and a metric tensor $\mathbf{M}$ and mesh metric are useful theoretical tools
for understanding a method's level of mesh alignment, or anisotropy. We
consider a mesh redistribution method based on the Monge-Amp\`ere
equation \cite{Finn},\cite{BW:06}, \cite{BW:09},
\cite{BHRActa}, \cite{Budd96}, which combines equidistribution of a given
scalar density function $\rho$ with optimal transport. It does not involve
explicit use of a metric tensor $\mathbf{M}$, although such a tensor must
exist for the method, and an interesting question to ask is whether
or not the alignment produced by the metric gives an anisotropic mesh.
For model problems with a linear feature and with a radially symmetric feature,
we derive the exact form of the metric $\mathbf{M}$, which involves
expressions for its eigenvalues and eigenvectors. The eigenvectors are shown
to be orthogonal and tangential to the feature, and the ratio of the
eigenvalues (corresponding to the level of anisotropy) is shown to depend,
both locally and globally, on the value of $\rho=\sqrt{\det{\mathbf{M}}}$
and the amount of curvature. We thereby demonstrate how the optimal
transport method produces an anisotropic mesh along a given feature
while equidistributing a suitably chosen scalar density function.
Numerical results are given to verify these results and to demonstrate
how the analysis is useful for problems involving more complex features,
including for a non-trivial time dependant nonlinear PDE which evolves narrow
and curved reaction fronts.
\end{abstract}
 
 
 
\pagestyle{myheadings}
\thispagestyle{plain}
\markboth{}{}
\section{Introduction}
 
\noindent
Efficiently and accurately computing solutions to PDEs (partial differential
equations) which exhibit large variations in small regions of a physical
domain frequently demands using some form of mesh adaptation/redistribution.
It is often desirable to adjust the size, shape and orientation of the
mesh elements to the geometry and flow field of the solution of the
underlying physical problem. More specifically, if the solution displays
anisotropic behaviour, then an anisotropic mesh can potentially capture
solution features with a minimal number of mesh points concentrated along such
features. This is in contrast to many adaptive methods, such as Winslow's
method \cite{Wins}, which explicitly adjust only the size of mesh elements,
typically using equidistribution of some measure of the solution
as a guide, and as a result often enforcing unnecessary shape regularity.

\vspace{0.1in}
 
\noindent As a consequence, there has been considerable interest in
designing adaptive mesh algorithms tailored for anisotropic problems. The
idea of using a metric tensor to quantify anisotropy was exploited in
two-dimensional mesh generation as early as the 1990's \cite{Fortin},
\cite{Az}, and accurate a posteriori \cite{Pic}, \cite{H2010}, and a priori
\cite{Form}, \cite{H2005}, anisotropic error estimates have since been
developed. For example, the Hessian matrix of a function provides a metric
\cite{Hecht} which arises in bounding error estimates for its interpolation error
and can be used to generate a mesh minimising this error \cite{Hecht2},
\cite{Cao}, \cite{Vallet}, \cite{H2005b}. Anisotropic mesh adaptation
methods have since been applied with great success to various problems
\cite{H2010b}, \cite{Fortin}, \cite{Form2}, \cite{Az2}, and much
software, such as BAMG \cite{Hecht3}, and Mesh Adap \cite{Li},
has been developed based on the metric tensor concept. The majority
of the codes implement adaptive mesh {\em refinement} (AMR or h-adaptivity)
methods in which meshes are locally refined by the addition of extra points.
Advantages of this approach are that the resulting methods
are flexible and robust and can deal with many complex solution and
boundary geometries; disadvantages are that h-adaptive methods
have complex data structures and refinement is predominately local,
which complicates understanding of global mesh regularity. Another
disadvantage is that when components of the flow move (e.g. eddies, fronts,
gravity currents), mesh points must be removed from regions they
have left and new mesh points included in the regions they enter.
As small-scale features propagate out of regions in which they are
resolved into regions in which they are partially resolved, this can
potentially lead to abrupt changes in grid resolution and result in
spurious wave reflection, refraction, or scattering \cite{Vich}, \cite{Vich2}.

\vspace{0.1in}
 
\noindent
In contrast, adaptive mesh {\em redistribution} methods, or r-adaptive
methods, relocate a fixed number of mesh points in an attempt to
generate an optimal mesh on which to represent the solution to the problem,
usually guided by the explicit or implicit construction of a mesh mapping and
a scalar or tensor valued monitor function represented in terms of the
Jacobian matrix of this mapping \cite{HRBook}. These methods potentially
offer certain advantages, such as fixed data structures, smoothly  
graded meshes, and an ability to
analyse through this mesh mapping a close coupling between the mesh and the
problem solution \cite{BHRActa}. Although still much less developed than AMR methods,
both theoretically and practically, they have been applied in many areas of
science and engineering with great success to solve problems involving
boundary layers, inversion layers, shock waves, ignition fronts, storm
fronts, gas combustion and groundwater hydrodynamics \cite{walsh},
\cite{HZZ}, \cite{Hyman}, \cite{stockie}, \cite{Tang},
\cite{TangTang}.
 
\vspace{0.1in}
 
\noindent
Anisotropic mesh generation for r-adaptivity is rigorously studied in
\cite{HRBook}, where a metric tensor (a symmetric positive definite
matrix valued monitor function) based on interpolation error
is derived. By showing the equivalence between a mesh constructed
from this metric tensor and certain equidistribution and alignment
conditions, one arrives at a good understanding of the geometry of the
resulting meshes. This metric tensor is closely tied to the
Jacobian of the associated mesh mapping. The majority of r-adaptive
methods considered in \cite{HRBook} use a variational approach, and various
classes of such methods are examined there, including ones involving
a combination of terms designed associated with equidistribution and alignment.
 
\vspace{0.1in}
 
\noindent In this paper we consider r-adaptive meshes generated from optimal
transport methods solving Monge-Ampere type problems. These methods,
described in \cite{Finn}, \cite{Chacon}, \cite{MohamRuss},
\cite{BW:06}, \cite{BW:09}, \cite{BHRActa}, \cite{walsh}, \cite{Pab},
combine local mesh scaling (equidistributing a specified scalar
monitor function to determine how big mesh elements
are) with a global regularity constraint (which requires that the mesh mapping
be as close as possible in a suitable norm to the identity mapping).
This requires solving an associated scalar Monge-Amp\`ere (MA) equation and
constructing the mesh mapping from the gradient of its solution.
These methods have the potential advantages of being robust, flexible,
and cheap to implement, for both two and three dimensional problems, 
particularly CFD type problems \cite{Cos1}, \cite{Cos2}.
They also have certain very desirable properties, such as
an absence of mesh tangling and an inheritance of self-similar
behaviour in the solution \cite{Finn},\cite{BW:09}. The above papers
describe in detail the implementation, convergence and scalability of
these methods to many examples. Interestingly, in an attempt to understand local and global properties of the mesh geometry 
analytical results have been obtained in \cite{Finn} that show optimal transport methods minimise a measure of grid distortion; 
however, to date analysis has been lacking for describing precise anisotropic structure of these meshes for sharp interfaces.
The main purpose of this paper is to provide such an analysis.
 
\vspace{0.1in}
 
\noindent
The mesh geometry can be described directly from the metric tensor, or
equivalently from the mesh qualities of
local scaling (mesh size), anisotropy (mesh alignment) and regularity
(mesh skewness) \cite{KNPP}, \cite{HRBook}. These are not entirely
straightforward to understand since a metric tensor is not used explicitly,
although it can be approximated as part of the mesh calculation. However,
in certain cases we can
deduce the local and global properties of the mesh from a careful study
of the analytic solutions of the associated (MA) equation. What is
discovered is that despite their being computed by equidistributing a
scalar quantity when solving the Monge-Amp\`ere equation, the meshes
generated in practice also show good alignment with sharp solution
features. More specifically, for model anisotropic problems having
solutions with linear features and with
high curvature features (including singularities), we are able to show
rigorously that even though the regularity condition imposed by optimal
transport is global, it also leads to anisotropic meshes closely aligned to
the features. The anaysis is
simplified by the fact that optimal transport methods give mesh mappings
with symmetric Jacobians, and consequently the alignment can be
simply related to their Jacobians. We see that the theoretical
results for the model problems are effective in predicting the mesh
behaviour (including the specific level of anisotropy) for more
complicated solutions to time dependent nonlinear PDEs. Moreover, the
results provide intriguing insight into a possible error analysis for
mesh adaptation methods based upon optimal transport.
 
\vspace{0.1in}
 
\noindent An outline of the paper is as follows. In
Section 2 we consider the basic principles of equidistribution and
alignment and the underlying optimal transport method. In Section 3 we examine
mesh alignment for problems with linear anisotropic solution features.
In Section 4 we provide a corresponding study for problems with
radially symmetric features with high curvature (singularities and rings).
In Section 5 we present two numerical examples, using the
results of Sections 3 and 4 to illustrate anisotropic mesh
properties for more complex nonlinear features.
The second example involves the solution of a nonlinear PDE
with an evolving front which is both narrow and has high curvature.
Final conclusions are given in Section 6.
 
\section{Basic principles of anisotropic mesh redistribution and the (MA) algorithm for mesh generation}

\noindent In this section we describe the basic features of r-adaptive mesh redistribution and the corresponding
description of the local mesh geometry in terms of a metric tensor. We then analyse an optimal transport algorithm in this context.

\subsection{ Mesh adaptation using a coordinate map}

\noindent An effective approach for studying the redistribution of an initially uniform mesh is to generate an invertible coordinate transformation $\mathbf{x}=\mathbf{x}(\boldsymbol{\xi}):\Omega_c\rightarrow \Omega_p$, from a fixed computational domain $\Omega_c$ to the physical domain  $\Omega_p$ in which the underlying PDE is posed \cite{HRBook}. The mesh $\tau_p$ in $\Omega_p$ is then generated as the image of a fixed uniform computational mesh $\tau_c$ in $\Omega_c$ which has a fixed number $N$ of elements of some prescribed shape. The alignment and other features of the mesh can then be determined by calculating the properties of the transformation $\mathbf{x}(\boldsymbol{\xi})$.
Assuming for the moment that $\mathbf{x}$ and $\boldsymbol{\xi}$ are given, and for simplicity restricting attention to the  2D case, 
we can consider the local properties of this transformation. 
Let $\hat{K}$ be a circular set in $\Omega_c$, centred at $\boldsymbol{\xi_0}$, such that
$$\hat{K} = \{ \boldsymbol{\xi}: (\boldsymbol{\xi}-\boldsymbol{\xi_0})^{T}(\boldsymbol{\xi}-\boldsymbol{\xi_0})=\hat{r}^2 \},$$
where the radius $\hat{r}\propto(\vert \Omega_c \vert / N)^{1/2}$. 
Linearizing  about  $\boldsymbol{\xi_0}$ we obtain
$$\mathbf{x}(\boldsymbol{\xi})=\mathbf{x}(\boldsymbol{\xi_0})+\mathbf{J}(\boldsymbol{\xi_0})(\boldsymbol{\xi}-\boldsymbol{\xi_0})+\mathrm{O}(\vert \boldsymbol{\xi}-\boldsymbol{\xi_0}\vert^2),$$ and
the corresponding image set $K=\mathbf{x}(\hat{K})$ in $\Omega_p$ is approximately given by
$${K} = \{ \mathbf{x}: (\mathbf{x}-\mathbf{x}(\boldsymbol{\xi_0}))^{T}\mathbf{J}^{-T}\mathbf{J}^{-1}(\mathbf{x}-\mathbf{x}(\boldsymbol{\xi_0}))=\hat{r}^2 \}.$$
As the set $K$ and $\boldsymbol{\xi_0}$ are arbitrary, we can replace $\boldsymbol{\xi_0}$ by a general point $\boldsymbol{\xi}$.
The Jacobian matrix $\mathbf{J}$ and its determinant $J$, referred to simply as the Jacobian, are 
\[
\mathbf{J}=\left[\begin{array}{cc} x_{\xi}& x_{\eta}\\ y_{\xi} & y_{\eta}\end{array}\right] \hspace{1cm} {J}=\left \vert \begin{array}{cc} x_{\xi}& x_{\eta}\\ y_{\xi} & y_{\eta}\end{array}\right \vert=x_{\xi}y_{\eta}-x_{\eta}y_{\xi}.
\]
Taking the singular value decomposition  
\[
\mathbf{J}=U\Sigma V^T, \hspace{1cm}\Sigma=\mathrm{diag}(\sigma_1,\sigma_2),
\]
it follows that
\[
{K} = \{ \mathbf{x}: (\mathbf{x}-\mathbf{x}(\boldsymbol{\xi_0}))^{T} \; U \; \Sigma^{-2} \;  U^T \; (\mathbf{x}-\mathbf{x}(\boldsymbol{\xi_0}))=\hat{r}^2 \}. 
\]
so that  the orientation of $K$ is determined by the left singular vectors  $U=[\mathbf{e_1},\mathbf{e_2}]$, and the size and shape by the singular values $\sigma_1$ and $\sigma_2$ (see Fig \ref{fig1:map}).
\begin{figure}[hhhhhhhh!!!!!]
\begin{center} \includegraphics[height=5cm,width=6cm]{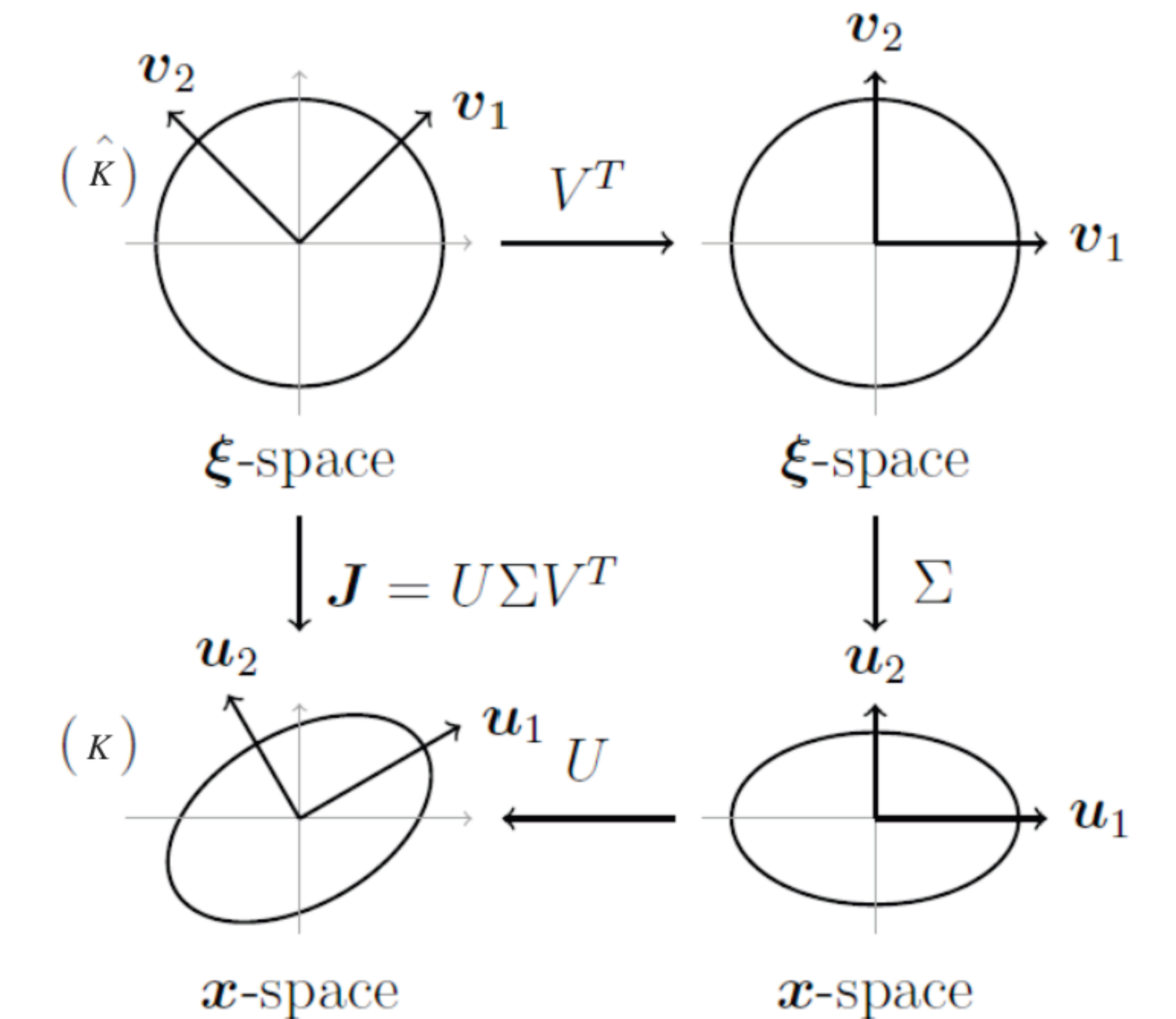}\end{center}
\caption{ \small The 2D mapping of a set  ($\hat{K}$, a circle) in $\Omega_c$, to a physical mesh element ($K$, an ellipse) in $\Omega_p$, under $\mathbf{x}(\boldsymbol{\xi})$. The local anisotropy of the transformation is evident from the degree of compression and stretching of the ellipse.}
\label{fig1:map}
\end{figure}
We can quantify the size, shape and orientation of an element $K$, in the continuous sense, using the singular values and left singular vectors of $\mathbf{J}$, 
and the eigenvalues and eigenvectors of the associated {\em metric tensor} 
\begin{equation}
\boldsymbol{\mathcal{M}}=\mathbf{J}^{-T}\mathbf{J}^{-1}. \label{eqal1}
\end{equation}
The eigenvectors of $\boldsymbol{\mathcal{M}}$ are $\textbf{e}_1$,$ \textbf{e}_2$ and the eigenvalues $\mu_1$, $\mu_2$,  satisfy
$\mu_i=1/\sigma_{i}^{2}$ for $ i=1,2$ , with 
\begin{eqnarray}
\boldsymbol{\mathcal{M}}=U\Sigma^{-2} U^T=\left[\begin{array}{cc} \mathbf{e_1}&  \mathbf{e_2}\end{array}\right]\left[\begin{array}{cc}\frac{1}{\sigma_1^2} & 0\\ 0 & \frac{1}{\sigma_2^2} \end{array}\right]\left[\begin{array}{c}  \mathbf{e_1}^T\\ \mathbf{e_2}^T\end{array}\right]. \nonumber
\end{eqnarray}
Hence, the circumscribed ellipse of a mesh element will have principal axes in the direction of the eigenvectors $\mathbf{e_1}$ and $\mathbf{e_2}$, with semi-lengths given by the  values $\sigma_1=\sqrt{1/\mu_1}$ and $\sigma_2=\sqrt{1/\mu_2}$, (although we note that in the discrete case the shape, size, and orientation of a mesh element are only partially determined by this metric). 
The anisotropy of the mesh locally is given by the ratio of $\sigma_1$ and $\sigma_2$. Accordingly, one natural way measure of the skewness $Q_s$ in terms of $\mathbf{J}$, which provides a measure of mesh quality, is 
\begin{equation}
Q_{s}=\frac{\mathrm{tr}(\mathbf{J^TJ})}{2\det(\mathbf{J^TJ})^{1/2}}=\frac{\sigma_1^2+\sigma_2^2}{2{\sigma_1 \sigma_2}}=\frac{1}{2}\left( {\frac{\sigma_1}{\sigma_2}}+  {\frac{\sigma_2}{\sigma_1}}\right).
\label{qske}
\end{equation}
This measure, and the circumscribed ellipse of a mesh element, are extremely useful for visualising and analysing the degree of anisotropy \cite{HRBook}, as we demonstrate later. We note that many other mesh quality measures exist, for example, those that take in to account small angles \cite{KNPP}, \cite{H2005b}, and more global measures of mesh quality such as the Kwok Chen metric \cite{KC00}. 

\subsection{Equidistribution and Alignment}

One approach to mesh adaptation is to equidistribute a  {\em scalar density function} $\rho(\mathbf{x})>0,$ over each mesh cell such that
\begin{equation}
\rho J=\theta.
\label{eqal2}
\end{equation}
where
\begin{equation}
\theta=\int_{\Omega_p}{\rho}  \; d\mathbf{x}/ \int_{\Omega_c}{d}\boldsymbol{\xi}.\label{thetadef}
\end{equation}
Equation (\ref{eqal2}) is the well known {\em equidistribution principle} which plays a fundamental role in mesh adaptation, giving direct control over the {\em size}, but not the alignment, of the mesh elements. For one-dimensional mesh generation it uniquely specifies the mesh and is widely used \cite{HRBook}, with prescribed $\rho$ often
given by some estimate of the solution error.

\vspace{0.1in}

\noindent For mesh generation in two or more dimensions the equidistribution principle (\ref{eqal2}) alone is insufficient to determine the  mesh uniquely and additional constraints are required \cite{Simp}.  Methods that augment the equidistribution principle with further local constraints are in
\cite{Baines}, \cite{Simp2}, \cite{H2001}, \cite{H2005}, \cite{Knupp}, and 
other principles for anisotropic mesh adaptation in \cite{steinberg}, \cite{BW:09}, \cite{Finn}.  A  common approach to locally controlled anisotropic mesh generation is to 
define the desired level of anisotropy through a metric tensor  $\mathrm{\textbf{M}}$ directly.
Then ${\mathbf M}$ is prescribed and the Jacobian ${\mathbf J}$ of the map is calculated {\bf directly} by enforcing the condition 
 \begin{equation}
Q_{a} \equiv \frac{\mathrm{tr}(\mathbf{J}^{T}\mathbf{M}\mathbf{J})}{2\det(\mathbf{J}^{T}\mathbf{M}\mathbf{J})^{1/2}} = 1.\label{eqal3}
\end{equation}
This extends the skewness measure (\ref{qske}) and is referred to  as the {\em alignment condition} \cite{HRBook}. As it requires that all elements are equilateral with respect to the metric $\mathrm{\textbf{M}}$ 
it allows for direct control of the shape and orientation of a mesh element through an appropriate choice of $\mathrm{\textbf{M}}$. It follows from (\ref{eqal1}) that for any scaled metric tensor $\mathrm{\textbf{M}}= \theta \boldsymbol{\mathcal{M}}$, that $\sqrt{\det(\mathrm{\textbf{M}})} J = \theta, $ for all $\mathbf{x}\in\Omega_p$, which by (\ref{eqal2}) is equidistribution of a \emph{scalar density function} 
\begin{equation}
\rho=\sqrt{\det(\mathrm{\textbf{M}})}.\label{sdf}
\end{equation}
Huang \cite{H2001} shows that combining the equidistribution and alignment conditions (\ref{eqal2})-(\ref{eqal3}) gives 
\begin{equation}
\mathbf{J}^{-T}\mathbf{J}^{-1}=\theta^{-1}\mathrm{\textbf{M}},\hspace{.2cm}\mbox{or equivalently}\hspace{.2cm}\mathbf{J}^{T}\mathrm{\textbf{M}}\mathbf{J}=\theta I.\label{eqal4a}
\end{equation}
That is, when the coordinate transformation satisfies relation (\ref{eqal4a}), the element size, shape, and orientation are completely determined by $\mathrm{\textbf{M}}$ throughout the domain. The resulting mesh will be aligned to the metric $\mathbf{M}$ and equidistributed with respect to the measure $\rho$, and is referred to as {\em M-uniform} \cite{HRBook}. In general there is no unique solution to (\ref{eqal4a}) for an apriori given $\mathbf{M}$, and so in practice this condition can only be enforced approximately. The choice of an appropriate metric tensor is important to the success of this method, and typically those which lead to low interpolation 
errors are chosen.
The simplest choice is to take a scalar matrix monitor function of the form 
\begin{equation}
\mathrm{\mathbf{M}}=\rho I.\label{smm}
\end{equation}
Using a variational approach this is equivalent to Winslow's variable diffusion method \cite{Wins}. 
In this case, 
the singular values of ${\mathbf M}$, and hence the semi-lengths of the circumscribed ellipse of a mesh element are equal (i.e., it is a circle) if
(\ref{smm}) is exactly satisfied and the corresponding mesh is isotropic.
In contrast, Huang \cite{H2005} has derived the exact forms of $\mathrm{\textbf{M}}$ for which the resulting mesh minimizes the interpolation error of some underlying function $u$. Piecewise constant interpolation error can be minimised in the L$^2$-norm if 
\begin{equation}
\mathrm{\textbf{M}}=\kappa_{h,1}[I+\alpha_{h,1}^2 \nabla u\nabla u^T]\label{eqal5}
\end{equation}
where 
$\alpha_{h,1}, \kappa_{h,1}$ are explicitly given parameters. For piecewise linear interpolation, the optimal
metric tensor is given by
\begin{equation}
\mathrm{\textbf{M}}=\kappa_{h,2}[I+\alpha_{h,2} H(u) ],\label{eqal6}
\end{equation}
for suitable parameters $\kappa_{h,2}$, and $\alpha_{h,2}$, where $H(u)$ is the Hessian matrix of $u$.

\vspace{0.1in}

\noindent Whilst effective in generating (essentially optimal) anisotropic meshes, these methods require finding the full Jacobian of the map at each step, which necessitates incorporating
extra convexity conditions to ensure uniqueness, making the resulting (typically variational) methods
challenging to implement. While a scalar matrix monitor function is simpler it can be too restrictive to produce a mesh that is aligned to a physical solution \cite{HRBook}. This begs the question of whether a method that equidistributes a scalar mesh density function is generally capable of producing anisotropic meshes. We demonstrate in the next section that by  combining equidistribution of a \emph{scalar density function} (\ref{sdf}) with a {\em global} constraint, namely \emph{optimal transport}, we can produce suitable anisotropic meshes which are relatively easy to compute.
Furthermore, for certain features, we are able to derive analytically the precise form of the metric $\mathbf{M}$ to which these meshes align and show it has a similar form to those metrics  given in (\ref{eqal5}) and (\ref{eqal6}) which minimise interpolation error. 


\subsection{Mesh redistribution using global constraints and the Monge-Amp\`ere equation}

\noindent  In contrast to the previous approaches we now augment condition (\ref{eqal2}) with {\em global constraints} to define the mesh, in particular {\em Optimal Transport Regularisation}
We seek to find a mesh mapping, satisfying (\ref{eqal2}), which is as close as possible (in a suitable norm) to the identity.

\vspace{0.1in}

\begin{definition}
An optimally equidistributed mapping $\mathbf{x}({\mathbf \xi})$ is one which minimizes the functional $I_{2}$, where
\[
I_2 =\int_{\Omega_c}\vert \mathbf{x} (\boldsymbol{\xi}) - \boldsymbol{\xi}\vert^2d\mathbf{x},
\]
over all invertible $\mathbf{x} (\boldsymbol{\xi})$ for which the equidistribution condition (\ref{eqal2}) also holds.
\end{definition}\\

\noindent The following result gives both the existence and uniqueness of such a map and a means to calculate it.

\vspace{0.1in}

\noindent \begin{theorem} (Brenier \cite{Bren}, Caffarelli \cite{Caff}) There exists a unique optimal
mapping $\mathbf{x}(\boldsymbol{\xi})$ satisfying the equidistribution condition (\ref{eqal2}). This map has the same regularity as $\rho$. Furthermore,
the map $\mathbf{x}(\boldsymbol{\xi})$ can be written as the gradient
(with respect to $\boldsymbol{\xi}$) of a unique (up to constants) convex mesh potential $P(\boldsymbol{\xi}, t)$, so that
\[
\mathbf{x}( \boldsymbol{\xi}) = \nabla_{\xi}P(\boldsymbol{\xi}), \hspace{1cm} \Delta_{\xi}P(\boldsymbol{\xi}) > 0.
\]
\end{theorem}

\vspace{0.1in}

\noindent It is immediate that if $\mathbf{x} = \nabla_{\xi}P$ then the Jacobian matrix ${\mathbf J}$ is {\em symmetric} and is the Hessian matrix of $P$, i.e. in two-dimensons
\begin{eqnarray}
\mathbf{J}=\mathbf{J}^T=\left[\begin{array}{cc} x_{\xi}& x_{\eta}\\ y_{\xi} & y_{\eta}\end{array}\right] =\left[\begin{array}{cc} P_{\xi\xi}& P_{\xi\eta}\\ P_{\eta\xi} & P_{\eta\eta}\end{array}\right]=:\mathbf{H}(P). \nonumber
\end{eqnarray}
Furthermore, the Jacobian determinant $ J$ is the Hessian determinant of $P$ such that in two-dimensional problems
\[
{J}=x_{\xi}y_{\eta}-x_{\eta}y_{\xi}=P_{\xi\xi}P_{\eta\eta}-P_{\xi\eta}^2:=H(P).
\]
The equidistribution condition (\ref{eqal2}) thus becomes 
\begin{equation}
\rho(\nabla P)  H(P) =\theta, \label{pmaeqa1}
\end{equation}
which is the  {\em Monge-Amp\`ere equation} (MA).  This fully nonlinear equation is generally augmented with Neumann or periodic boundary conditions, where the boundary of $\Omega_c$
is mapped to the boundary of $\Omega_p$ \cite{BW:06}, \cite{Finn}. However solutions have also been attained for non-standard boundary conditions \cite{froese2} and so it has the potential to be applied to more complex geometries. 
The gradient of $P$  thereby gives the unique map $\mathbf{x}$.
Methods to solve (\ref{pmaeqa1}) are described in \cite{BW:09},\cite{Finn}, \cite{cos1}, and form the basis of effective and robust mesh redistribution algorithms in two and three dimensions \cite{Pab}.
These methods have several advantages in practical applications. In particular, they only involve solving scalar equations, they deal naturally with
complex boundaries, and they can be easily coupled to existing software both for solving certain PDEs \cite{BW:06}, \cite{walsh} (see also Section 5) and also for approximating functions in operational data assimilation  codes \cite{culpic}.

\vspace{0.1in}

\noindent While these meshes satisfy the local scaling condition (\ref{eqal2}), regions where $\rho$ is large will result in small mesh elements and vice versa. However, it is not immediately clear what 
shape and orientation the elements inherit from (\ref{pmaeqa1}), although in \cite{Finn} it is shown these meshes minimise the global distortion as measured by the integral of 
\[
\mathrm{tr}(\mathbf{J^TJ})=\sigma_1^2+\sigma_2^2.
\]
We study this further here by seeking exact solutions of (\ref{pmaeqa1}) and the corresponding meshes.
To do this we use the following result:

\vspace{0.1in}

\begin{lemma} For a given scalar function $\rho(\mathbf{x})$, the solution of (\ref{pmaeqa1}) is unique, and the corresponding mesh has a unique metric tensor 
${\mathbf M}$, for which 
$$\rho = \sqrt{\det({\mathbf M})}.$$
\end{lemma}

 \begin{proof} Given $\rho(\mathbf{x})$, it follows from Theorem 2.2 that the Monge-Amp\'ere equation (\ref{pmaeqa1}) has a unique solution $P$. Hence we may uniquely construct
the Jacobian matrix  ${\mathbf J} = {\mathbf H}(P)$ and metric tensor ${\mathbf M} = \theta {\mathbf J}^{-1} {\mathbf J}^{-T}.$ Since 
$J \sqrt{\det(M)} = \theta = \rho J$ from (\ref{pmaeqa1}), the result follows. \qquad\end{proof}

\vspace{0.1in}

\noindent We can calculate the explicit form of ${\mathbf M}$ as follows: Assume that we are considering problems in $R^n$.  Since ${\mathbf J}$ is {\em symmetric} its eigenvalues $\lambda_1,\lambda_2, \ldots , \lambda_n$ are equal to its singular values 
$\sigma_1,\sigma_2, \ldots, \sigma_n$ and
its (unit) eigenvectors ${\mathbf e}_1, {\mathbf e}_2 \ldots {\mathbf e}_n$ are orthogonal. The Jacobian can therefore be expressed in the form
$$
{\mathbf J} = \lambda_1 {\mathbf e}_1 {\mathbf e}_1^T + \lambda_2 {\mathbf e}_2 {\mathbf e}_2 ^T + \ldots +  \lambda_n {\mathbf e}_n {\mathbf e}_n ^T 
$$
implying $\rho={\theta}/{J}={\theta}/{\lambda_1 \lambda_2 \ldots \lambda_n}.$ It follows from (\ref{eqal4a}) that  the metric tensor ${\mathbf M}$ for which the mesh is M-uniform has the same (unit) orthogonal eigenvectors ${\mathbf e}_i$ 
and eigenvalues $\mu_i = {\theta}/{\lambda_i^2}$
and can be expressed in the form 
\begin{equation}
{\mathbf M} = \theta \left( \lambda_1^{-2}  {\mathbf e}_1 {\mathbf e}_1^T + \lambda_2^{-2}  {\mathbf e}_2 {\mathbf e}_2 ^T  + \ldots +  \lambda_n^{-2}  {\mathbf e}_n {\mathbf e}_n ^T \right).
\label{cbjac2}
\end{equation}
Observe that this metric tensor is not generally a scalar multiple of the identity matrix and differs from the Jacobian.


\section{Alignment to a linear feature}

\noindent In this section we consider how well the meshes generated by solving (\ref{pmaeqa1}) represent two-dimensional linear features, looking at the alignment, scaling, skewness and anisotropy of the
meshes constructed for both single shocks and for shocks meeting
orthogonally. These are prototypes of the more complex forms of shocks and fronts found in applications \cite{Tang},\cite{walsh}. Our study will
centre on certain exact solutions of (\ref{pmaeqa1}). To obtain these solutions we will consider simple domains with periodic boundary conditions.
Whilst clearly not representative of many applications, we can still use the results obtained as a good local description of the mesh close
to linear regions of more complex features in a more complex geometry.

\subsection{Construction of an exact map} 
Let the scalar density $\rho(\mathbf{x})$ take the form
\begin{eqnarray}
\rho(\mathbf{x})=\rho_1(\mathbf{x}\cdot \mathbf{e_1})\rho_2(\mathbf{x}\cdot \mathbf{e_2}):=\rho_1(x^{\prime})\rho_2(y^{\prime}).\label{rhoc}
\end{eqnarray}
where $\mathbf{e_1}=[a,b]^T, \quad 
\mathbf{e_2}=[-b, a]^T, \quad a^2 + b^2 = 1.$
Furthermore, assume that the periodic function $\rho_1$ is large when $\mathbf{x}\cdot \mathbf{e_1}=c$, and the periodic function $\rho_2$ is large when $\mathbf{x}\cdot \mathbf{e_2}=d$,  for given constants $c$, and  $d$, and that they are close  to $1$ otherwise.
Note that the solution of the equidistribution equation (\ref{eqal1}) would be expected to concentrate mesh points 
along the lines given by either of the conditions  $\mathbf{x}\cdot \mathbf{e_1}=c$, or $\mathbf{x}\cdot \mathbf{e_2}=d$. 

\vspace{0.1in}

\begin{theorem} If the scalar density $\rho(\mathbf{x})$ has the form given in (\ref{rhoc})
then the Monge Ampere equation can be solved exactly in a doubly periodic domain. For the 
resulting mapping  the uniquely derived metric tensor ${\mathbf M}$ satisfies
(\ref{cbjac2}), and the mesh aligns exactly along the linear features.
\end{theorem}

\vspace{0.2in}
 \begin{proof} To show this result we consider the case where $\Omega_c = \Omega_p = (0,1)^2$ and the solution to (\ref{pmaeqa1}) 
is a {\em doubly-periodic} map from $\Omega_c\rightarrow\Omega_p$, such that $\boldsymbol{\xi}=[\xi,\eta]\in\Omega_c$, $\mathbf{x}=[x, y]\in\Omega_p$. 
The value of  $\theta$ defined in (\ref{thetadef}) is calculated as below.
\\
\\ 
  \begin{lemma} 
If $\theta_1$ and $\theta_2$ are defined as follows
\begin{equation}
{
 \theta_1 = \int_0^1 \rho_1(s) \; ds, \quad \mbox{and} \quad \theta_2 = \int_0^1 \rho_2(s) \; ds. 
}
\label{theta1}
\end{equation}
then $\theta = \theta_1 \theta_2$.
\end{lemma}

\vspace{0.1in}

 \begin{proof} By the definition in (\ref{thetadef}) 
$$\theta = \int_{\Omega_p} \rho({\mathbf x})  \; dx/\int_{\Omega_c} \; d\xi = \int_0^1 \int_0^1 \rho_1({\mathbf x}\cdot{\mathbf e}_1) \rho_2({\mathbf x}\cdot{\mathbf e}_2) \; dx dy/ \int_0^1 \int_0^1 \; d\xi d\eta.$$
Introducing coordinates $x' = {\mathbf x}\cdot{\mathbf e}_1 and \quad y'={\mathbf x}\cdot{\mathbf e}_2$, since $\mathbf{e_1}$ and $\mathbf{e_2}$ are orthonormal it follows immediately that $dx \; dy = dx' \; dy'$, so from double-periodicity of $\rho$  we have
$$\theta = \int_0^1 \int_0^1 \rho_1(x') \rho_2(y') \; dx' dy' =  \int_0^1 \rho_1(x') \; dx'  \int_0^1 \rho_2(y') \; dy' = \theta_1 \theta_2.$$ \qquad\end{proof}

\vspace{0.1in}

\noindent It follows that the Monge-Amp\`ere equation (\ref{pmaeqa1}) can be expressed in the form
\begin{eqnarray}
H(P) \;  \rho_1(x^{\prime}) \rho_2(y^{\prime}) = \theta_1 \theta_2. \label{MAls}
\end{eqnarray}
Fortuitously, this fully nonlinear PDE is separable and has an exact solution, from which we can calculate the mesh, the metric tensor and the skewness $Q_s$.

\vspace{0.2in}

\begin{lemma}For appropriate periodic functions $F(t)$ and $G(t)$ given by the solution of (\ref{EVfg}), there exists a doubly-periodic, separable solution to (\ref{MAls}) of the form
\begin{equation}
P(\xi,\eta)=F(\boldsymbol{\xi}\cdot\mathbf{e_1})+G(\boldsymbol{\xi}\cdot\mathbf{e_2}).\label{VSS}
\end{equation}
Furthermore, this solution is unique up to an arbitrary constant of addition.
\end{lemma}

\vspace{0.1in}

\begin{proof} Differentiating (\ref{VSS}) with respect to $\xi$ and $\eta$ gives
\begin{eqnarray}
\mathbf{x}=\nabla_{\xi}P=\mathbf{e_1}^TF^{\prime}+\mathbf{e_2}^TG^{\prime}.\label{mapls}
\end{eqnarray}
Differentiating again with respect to $\xi$ and $\eta$ we obtain
\begin{equation}
P_{\xi\xi} = a^2 F^{\prime\prime}+b^2 G^{\prime\prime}, \quad 
P_{\xi\eta } = ab F^{\prime\prime}-ab G^{\prime\prime}, \quad
P_{\eta\eta} = b^2 F^{\prime\prime}+a^2 G^{\prime\prime}.\nonumber
\end{equation}
Hence 
\[
\mathbf{H}(P) =\left[\begin{array}{cc}\mathbf{e_1}&\mathbf{e_2} \end{array}\right]\left[\begin{array}{cc}F^{\prime\prime}&0 \\ 0&G^{\prime\prime}\end{array}\right]\left[\begin{array}{c}\mathbf{e_1}^T\\ \mathbf{e_2}^T\end{array}\right]
\]
and
\begin{eqnarray}
H(P)&=&(a^2F^{\prime\prime}+b^2G^{\prime\prime})(b^2F^{\prime\prime}+a^2G^{\prime\prime})-(abF^{\prime\prime}-abG^{\prime\prime})^2\nonumber\\
&=&(b^2+a^2)^2F^{\prime\prime}G^{\prime\prime}=F^{\prime\prime}G^{\prime\prime}.\label{HessVSS}
\end{eqnarray}
Substituting (\ref{HessVSS}) into the Monge-Amp\`ere equation (\ref{MAls}) we obtain
\[
F^{\prime\prime}(\xi^{\prime})G^{\prime\prime}(\eta^{\prime}) \; \rho_1(x') \rho_2(y') = \theta_1 \theta_2,
\]
where $\xi' = \boldsymbol{\xi}\cdot\mathbf{e_1}$ and $\eta' = \boldsymbol{\xi}\cdot\mathbf{e_2}.$
Now by (\ref{mapls}) it follows that
\[
x^{\prime}=\mathbf{x}\cdot \mathbf{e_1}=\mathbf{e_1}^T\cdot \mathbf{e_1}F^{\prime}+\mathbf{e_2}^T\cdot \mathbf{e_1}G^{\prime}=F^{\prime}(\xi^{\prime}),
y^{\prime}=\mathbf{x}\cdot \mathbf{e_2}=\mathbf{e_1}^T\cdot \mathbf{e_2}F^{\prime}+\mathbf{e_2}^T\cdot \mathbf{e_2}G^{\prime}=G^{\prime}(\eta^{\prime}).
\]
Thus, there is a solution of (\ref{MAls}) of the form (\ref{VSS}) provided $F$ and $G$ satisfy 
\begin{equation}
{
F^{\prime\prime}(\xi^{\prime}) \rho_1(F'(\xi')) =\theta_1 \alpha \quad \mbox{and} \quad  G^{\prime\prime}(\eta^{\prime}) \rho_2(G'(\eta')) =\theta_2/\alpha, 
}
\label{EVfg}
\end{equation}
where $\alpha$ is (at this stage) an arbitrary constant. From the identities $x' = F'$ and $y' = G'$ it follows that 
$x'(\xi') \rho_1(x'(\xi')) = \theta_1 \alpha$ and for a suitable constant $c_1$, $R_1(x') \equiv \int_0^{x'}  \rho_1(s) \; ds = \theta_1 \alpha \;  \xi' + c_1.$
Since the map from $\Omega_c$ to $\Omega_p$  is doubly periodic, 
$x'(0) = 0$ and $x'(1) = 1$. Thus, $c_1 = 0$ and from the definition of $\theta_1$, $\alpha = 1$. 
Hence, we have
\begin{equation}
{
x' = {\mathbf x}\cdot{\mathbf e}_1 = R_1^{-1}(\theta_1 \; \xi') = R_1^{-1}(\theta_1 \; {\boldsymbol{\xi}}\cdot{\mathbf e}_1).
}
\label{xstuff}
\end{equation}
A similar identity follows for  $y'$ with related function $R_2$ and constant $c_2$, giving
\begin{equation}
{
 y' = {\mathbf x}\cdot{\mathbf e}_2 = R_2^{-1}(\theta_2 \; \eta') = R_2^{-1}(\theta_2 \; {\boldsymbol{\xi}}\cdot{\mathbf e}_2).
 }
 \label{ystuff}
\end{equation}
These define the functions $F$ and $G$, and
the uniqueness (\ref{VSS}) follows from the uniqueness of solutions of the Monge-Amp\`ere equation (\ref{MAls}) 
with periodic boundary conditions \cite{Loeper}. \qquad\end{proof}

\vspace{0.2in}

\noindent We now calculate the Jacobian of the map ${\mathbf J}$ and the metric tensor ${\mathbf M}$. Note that
$$
\mathbf{x}=\nabla_{\xi}P=\mathbf{e_1}^T R_1^{-1}(\theta_1 \xi^{\prime})+\mathbf{e_2}^T R_2^{-1}(\theta_2 \eta^{\prime})
$$
and 
\begin{equation}
{
{\mathbf J} = \frac{\theta_1}{\rho_1(F^\prime(\xi^\prime))}\; {\mathbf e_1}\;{\mathbf e_1^T} +\frac{\theta_2}{\rho_2(G^\prime(\eta^\prime))} \;{\mathbf e_2}\;{\mathbf e_2^T}
}
\label{jacex34}
\end{equation}
with eigen/singular values 
\begin{equation}
\lambda_1 =  \theta_1/\rho_1, \hspace{.2cm}\mbox{and}\hspace{.2cm} \lambda_2=\theta_2/\rho_2.
\label{spotty1}
\end{equation}
From (\ref{eqal4a}), we infer that the mesh will be aligned to the metric
\begin{equation}
{\mathbf M} = \frac{\theta_2  \rho_1^2 }{\theta_1} \; {\mathbf e_1}\;{\mathbf e_1^T} +\frac{\theta_1 \rho_2^2}{\theta_2} \;{\mathbf e_2}\;{\mathbf e_2^T},
\label{jacex35c}
\end{equation}
with eigenvalues
\begin{equation}
\mu_1 ={\theta_2 \rho_1^2}/{\theta_1} \quad \mbox{and} \quad \mu_2 ={\theta_1 \rho_2^2}/{\theta_2}.
\label{spotty2}
\end{equation}

\noindent These explicit forms for ${\mathbf J}$ and ${\mathbf M}$ reveal the alignment properties
of the map. Specifically, the eigendecomposition of ${\mathbf J}$ in (\ref{jacex34})
shows that the semi-axes of the ellipses described in Section 2  are parallel to
${\mathbf e_1}$ and ${\mathbf e_2}$ and thus align with the linear features.
The linear features we are aiming to represent arise when ${\mathbf x}\cdot{\mathbf e_1} = x' =  c$ and ${\mathbf x}\cdot{\mathbf e_2} = y' =  d$ so that 
respectively either $\rho_1$ is large and $\rho_2$ is not, or $\rho_2$  is large
and $\rho_1$ is not.  This ends the proof of Theorem 3.1. \qquad\end{proof}

\noindent We can also study the mesh away from the features.

\vspace{0.2in}

\begin{corollary}
Away from the linear features the mesh is in general isotropic and its skewness is given explicitly by 
\begin{equation}
Q_s = \frac{1}{2} \left( \frac{\theta_1 \rho_2}{\theta_2 \rho_1} + \frac{\theta_2 \rho_1}{\theta_1 \rho_2} \right) 
\label{cskew}
\end{equation}
\end{corollary}

\vspace{0.1in}

\noindent \begin{proof}  Substituting the expressions from our explicit solution into (\ref{qske}) gives
\begin{equation}
Q_s = \frac{1}{2} \left( \frac{\theta_1 \rho_2}{\theta_2 \rho_1} + \frac{\theta_2 \rho_1}{\theta_1 \rho_2} \right) 
\label{cskew1}
\end{equation}
The value of $Q_s$ depends upon  the relative size of the density functions $\rho_1$ and $\rho_2$, both locally and globally. \qquad\end{proof}

\noindent Along the linear features, where either $\rho_1 \gg1$ and $\rho_2={\cal O}(1)$, or $\rho_2 \gg 1$ and $\rho_1={\cal O}(1)$, the mesh elements will be anisotropic and skew. Away from the linear feature, where $\rho_1$ and $\rho_2$ are both of order one, the degree of anisotropy and skewness is controlled by the relative values of the density functions in the entire domain, $\theta_1$ and $\theta_2$. As these are averaged quantities the ratio is again in general of order one. We give precise estimates presently
for these in two examples below.

\subsection{Examples}
We now consider two specific analytical examples which illustrate the theory described above.
\subsubsection{Example 1: A single periodic shock}
As a first example we consider a periodic array of linear features aligned at $\pi/4$ to the coordinate axes so that ${\mathbf e_1}^T=(1 \; 1)/\sqrt{2}$
and ${\mathbf e}_2^T=(1 \; -1)/\sqrt{2}.$  As a periodic mesh density we take
\begin{equation}
\rho(\mathbf{x})=1+\alpha\sum_{n=-\infty}^{\infty} \mathrm{sech}^2(\alpha(\sqrt{2}x^{\prime}-n)):=\rho_1(x^{\prime}), \quad x' = {\mathbf x}\cdot{\mathbf e_1}, \label{ldf}
\end{equation}
with $\alpha=50$. This density is concentrated along a set of lines of width $1/50\sqrt{2}$ which are parallel to ${\mathbf e}_2$, one of which passes through the coordinate origin,
and the others arising when $x'=\pm 1/\sqrt{2}, \pm 2/\sqrt{2}, \ldots$. Note that along each such line $\rho = 51 + {\cal O}(\exp(-50))$ and away from each such line  $\rho = 1 + {\cal O}(\exp(-50)).$
A direct calculation gives
$$\theta=\int_{\Omega_p}\rho(\mathbf{x}) \;  d\mathbf{x}= 3 + {\cal O}(\exp(-50)),$$
and
$$
R_1(x^{\prime})=x^{\prime}+\frac{1}{\sqrt{2}}\sum_{n=-\infty}^{\infty}[\tanh(50(\sqrt{2}x^{\prime}-n))-\tanh(-50n)].
$$
The inverse of $R_1$ can be computed by fitting a spline through the data points $(R_1(x^{\prime}_i), \hspace{.2cm}x^{\prime}_i)$, for $x^{\prime}_i=\sqrt{2}i/N'$, $i=0,...,N'$. A plot of $R_1^{-1}$ is given in Fig. \ref{R1} for $N'=1000$. 
\begin{figure}[hhhhhhhhhhhhhhhhhhhhhhhhhhhhhhhhhhhhhhhhhhhhhh]
\begin{center}\includegraphics[height=4.5cm,width=5.5cm]{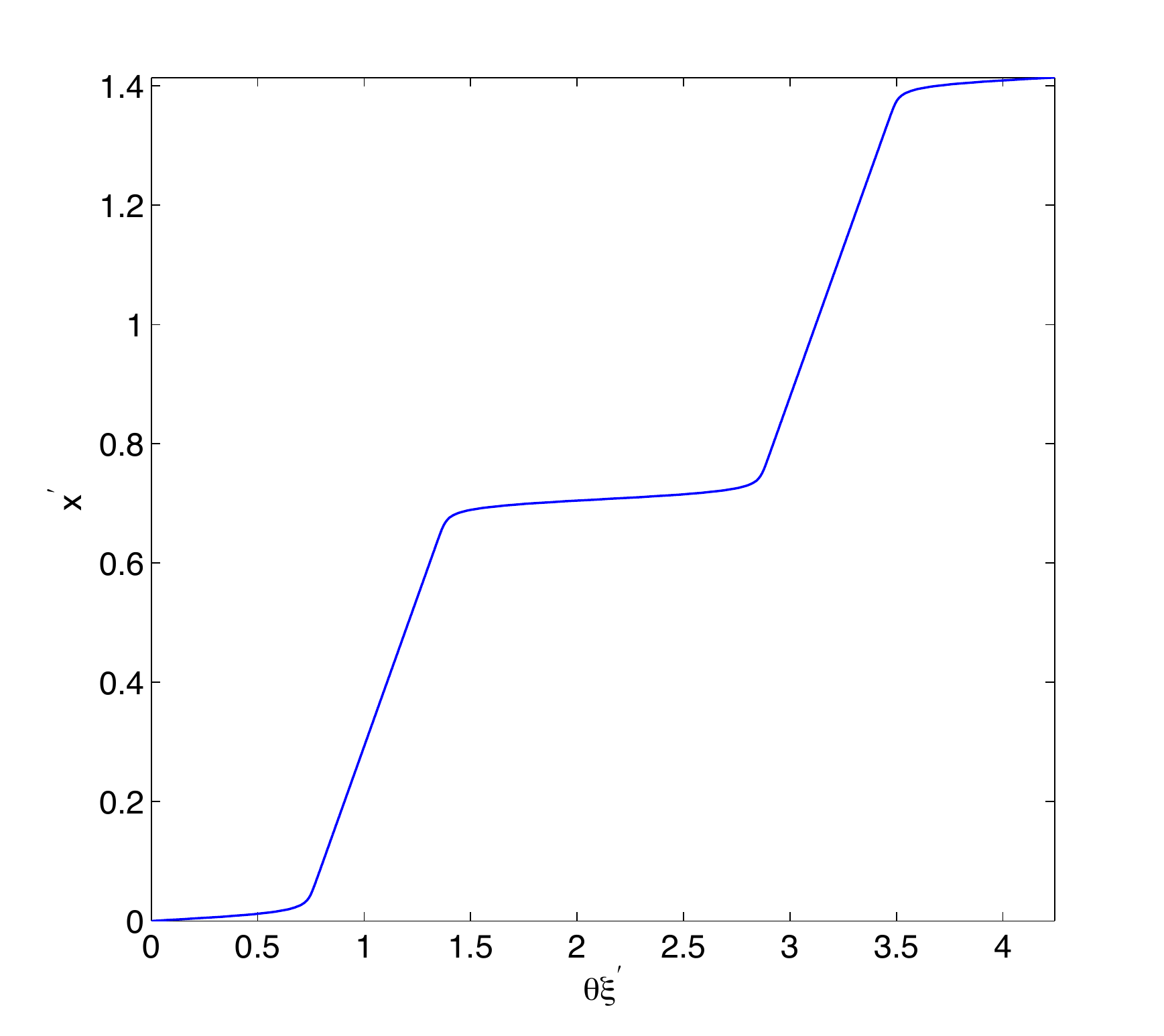}\end{center}
\caption{The function $R_1^{-1}$ for Case 1, $\theta=3 + {\cal O}(\exp(-50))$.}
\label{R1}
\end{figure}
Observe that this function is very flat close to $x'=0,1/\sqrt{2},\sqrt{2}$, and mesh points will be
concentrated at these values.
It follows immediately that  $\theta_1 = \theta, \theta_2 = 1, R_2(y') = y',$ and also
$$\xi' = (\xi + \eta)/\sqrt{2}, \quad \eta'=(\xi-\eta)/\sqrt{2}, \quad x = (x'+y')/\sqrt{2}, \hspace{.1cm}\mbox{and} \hspace{.1cm} y=(x'-y')/\sqrt{2}.$$
Therefore, from (\ref{xstuff}) and (\ref{ystuff}) it follows that
$$
x=\frac{1}{\sqrt{2}}[R_1^{-1}(\theta(\xi +\eta)/\sqrt{2})-((-\xi +\eta)/\sqrt{2})], \quad 
y=\frac{1}{\sqrt{2}}[R_1^{-1}(\theta(\xi +\eta)/\sqrt{2})+((-\xi +\eta)/\sqrt{2})].
$$
A plot of the resulting mesh is shown in Fig. \ref{shocke1}(a) with a close-up in Fig. \ref{shocke1}(b). This mesh is the image of a uniform square computational mesh and has the 
points $(x(\xi_i,\eta_j),y(\xi_i,\eta_j))$, where $\xi_j=\eta_j=j/(n-1)$, for $i,j=0,...,N-1$ and $N=60$.
\begin{figure}[hhhhhhhhhhhhhhhhhhhhhhhhhhhhhhhhhhhhhhhhhhhhhhhhhhhhhhhhhhhhhhhhhhhhhhhhhhhhhhhhhhhhhhhhhhhh]
\begin{center}\includegraphics[height=5.5cm,width=6cm]{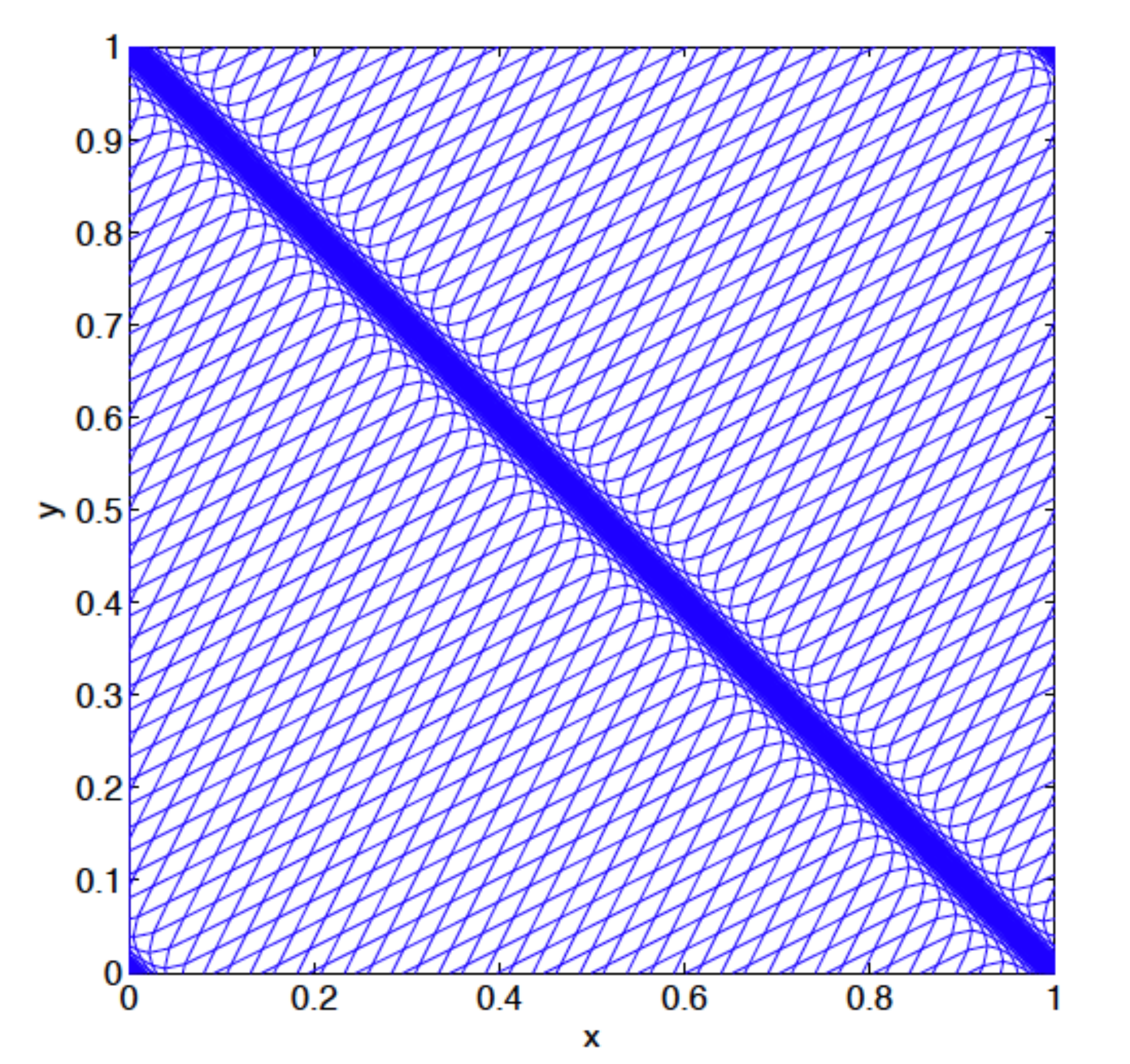}\includegraphics[height=5.5cm,width=6cm]{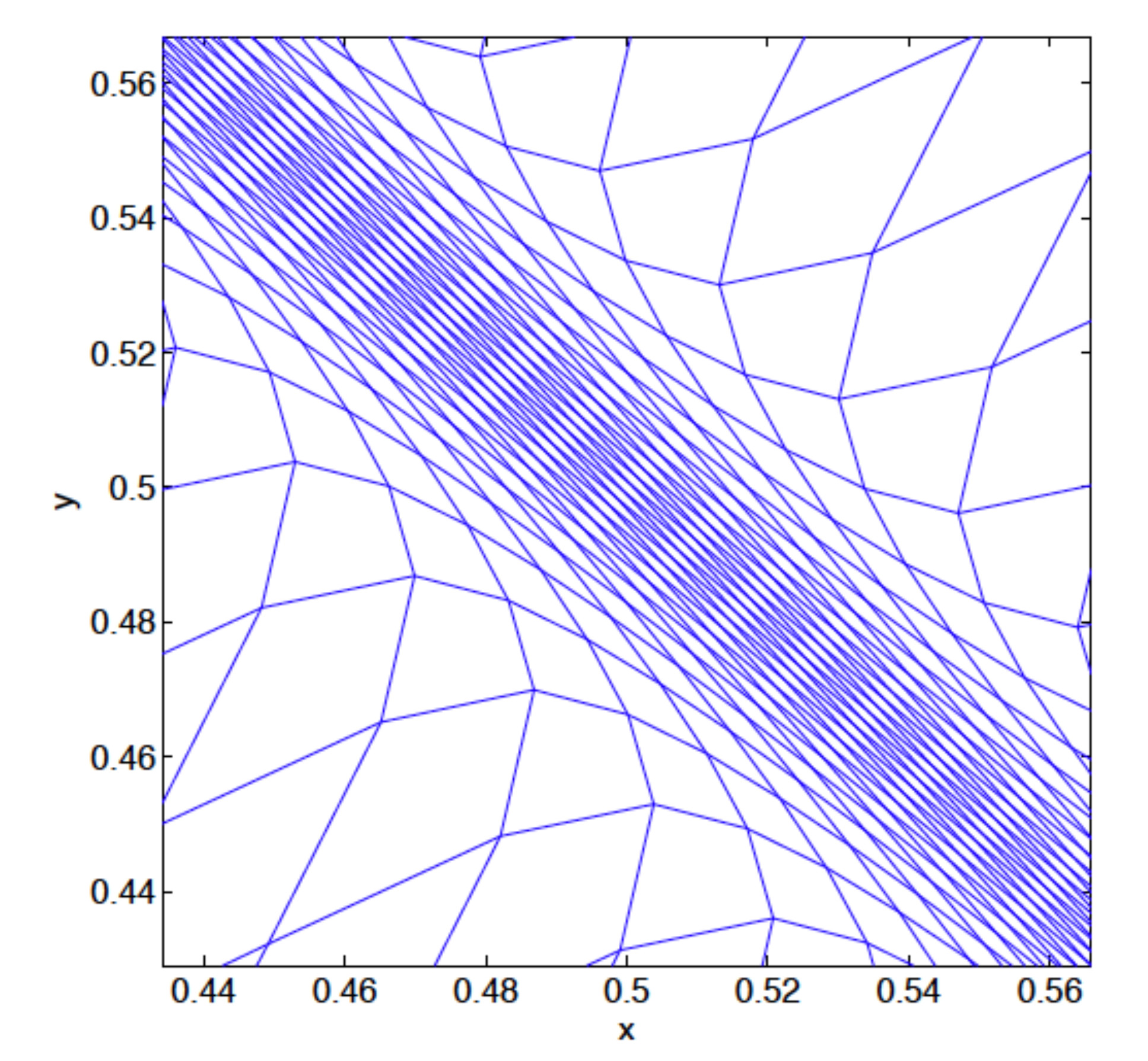}
\caption{(Left) A $(60\times60)$ mesh generated from the analytical solution of the Monge-Amp\`ere equation for the density function in Case 1. (Right) A zoom of the region along the shock where the density function is large.}\label{shocke1}\end{center}
\end{figure}
We see that not only is the mesh concentrated along the linear features parallel to ${\mathbf e}_2$ but it is also closely
aligned with this vector. Away from the linear feature the mesh has a distinctive diamond shape, with each diamond of uniform size and with axes  in the directions
${\mathbf e}_1$ and ${\mathbf e}_2$. The close-up shows the diamonds stretched along the linear feature and then smoothly evolving into 
uniform diamonds. The skewness of the mesh can be calculated directly from the Jacobian.  The eigenvalues of ${\mathbf J}$ (which coincide with the singular values) 
are given from (\ref{spotty1}) by $\lambda_1 = \theta/\rho$ and $\lambda_2 = 1$. Ignoring exponentially small terms, we have $\lambda_1 =  3/51$ within the linear feature, and $\lambda_1 = 3$
away from the linear feature, implying that the skewness measure $Q_s$ in (\ref{qske}) is given by
$Q_s = 8.529$ within the linear feature and $Q_s = 1.667$ outside the  linear feature. Although the specific example given here is not very anisotropic, extremely anisotropic meshes, whilst simple to compute, are difficult to visualise.
\noindent \begin{lemma} A mesh generated by solving (MA) (\ref{pmaeqa1}), with a density function of the form (\ref{ldf}), concentrates mesh points along a set of lines of width $\epsilon=1/\alpha\sqrt{2}$, where the mesh is anisotropic with skewness measure $Q_s$ (\ref{cskew}) inversely proportional to $\epsilon$.
\end{lemma}
\begin{proof} 
Ignoring exponentially small terms, the eigenvalues of the Jacobian of the mesh mapping where the density function is at a maximum are $\lambda_1=\theta/(1+\alpha)$ and $\lambda_2=1$, hence the skewness $Q_s=1/2((1+\alpha)/\theta+\theta/(1+\alpha))$.
Since $\alpha>>1$ and $\theta$ is order 1
\[
Q_s\approx \frac{1}{\sqrt{2}\theta \epsilon}.
\]
\end{proof}
\subsubsection{Example 2: Two orthogonal shocks}
Consider orthogonal shocks of different widths and magnitudes with the associated scalar density
$
\rho(\mathbf{x})=\rho_1(x^{\prime})\rho_2(y^{\prime}).
$
Here $\rho_1(x')$, $\theta_1$, and $R_1(x')$ 
are the same as in Example 1, and 
\[\rho_2=1+10\sum_{m=-\infty}^{\infty}\mathrm{sech}^2(25(\sqrt{2}y^{\prime}-m)).
\]
A direct calculation gives $\theta_2=1.8 + {\cal O}(\exp(-25))$, and
\[
R_2(y^{\prime})=y^{\prime}+\frac{\sqrt{2}}{5}\sum_{m=-\infty}^{\infty}[\tanh(25(\sqrt{2}y^{\prime}-m))-\tanh(-25m)].
\]
The inverse of $R_2$ can be computed in the same manner as for $R_1$ in the previous case.
Using the same procedures as in Example 1, we have
\begin{eqnarray}
x&=&\frac{1}{\sqrt{2}}[R_1^{-1}(\theta_1(\xi +\eta)/\sqrt{2})-R_2^{-1}(\theta_2(-\xi +\eta)/\sqrt{2})],\nonumber\\
y&=&\frac{1}{\sqrt{2}}[R_1^{-1}(\theta_1(\xi +\eta)/\sqrt{2})+R_2^{-1}(\theta_2(-\xi +\eta)/\sqrt{2})].\nonumber
\end{eqnarray}
 A plot of the image of a uniform mesh under this map is shown in Fig. \ref{shocke3},
\begin{figure}[hhhhhhhhhhhhhhhhhhhhhhhhhhhhhhhhhhhhhhhhhhhhhh]
\includegraphics[height=5.5cm,width=6cm]{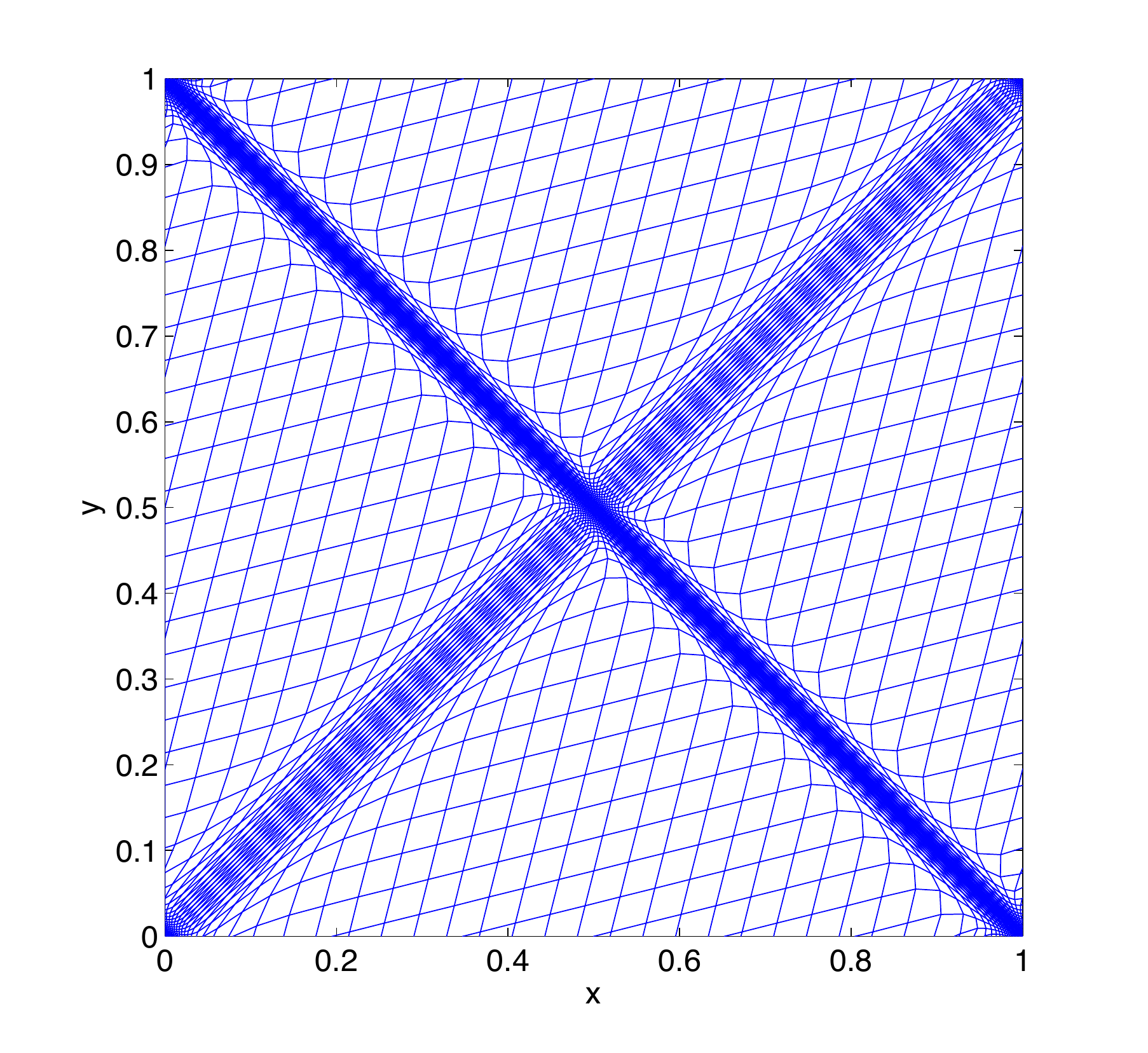}\includegraphics[height=5.5cm,width=7.2cm]{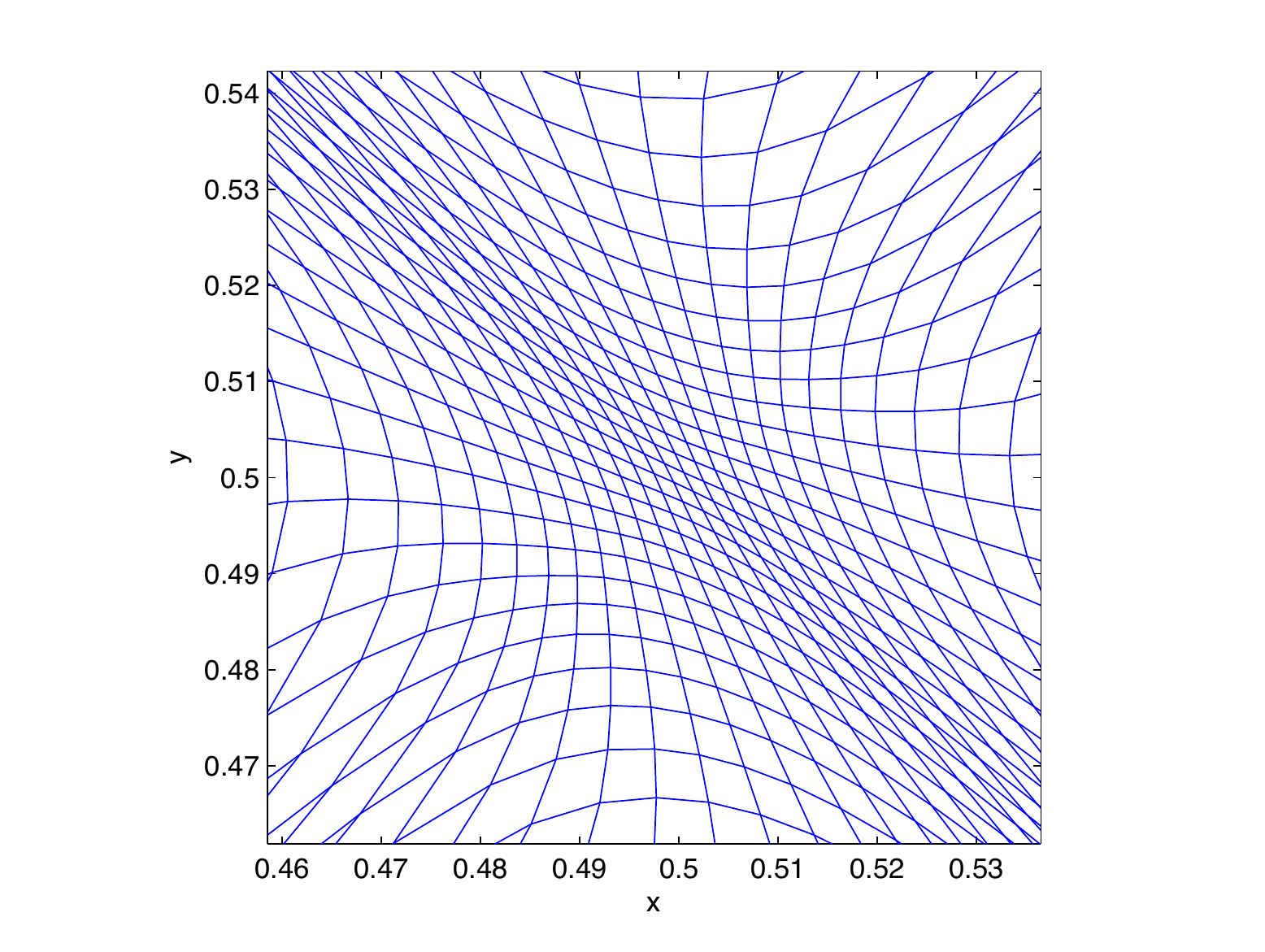}
\caption{(Left) A $(60\times60)$ mesh generated from the analytical solution of the Monge-Amp\`ere equation for the density function in Case 2. (Right) A zoom of the region along the shock where the density function is large.}
\label{shocke3}
\end{figure}
where we see the excellent alignment of the mesh to the two linear features. Note also the very smooth transition of the mesh
from one feature to the other. The eigenvalues  $\lambda_1, \lambda_2$ (up to exponentially small terms) are:
\begin{enumerate}
\item First linear feature alone: \quad $\lambda_1=3/51, \quad \lambda_2 = 1.8, $
\item Second linear feature alone: \quad $\lambda_1 = 3, \quad \lambda_2 = 1.8/11$
\item Intersection of the two linear features: $\lambda_1 = 3/51, \quad \lambda_2 = 1.8/11$
\item Outside the two linear features: $\lambda_1 = 3, \quad \lambda_2 = 1.8.$
\end{enumerate}
The respective values of the skewness measure $Q_s$ are 
$$1. \quad Q_s  = 15.31, \quad 2. \quad Q_s =  9.19, \quad 3. \quad Q_s = 1.57, \quad 4. \quad Q_s = 1.13.$$
We deduce that away from the linear features and also in the intersection of the two features the mesh in 
Example 2 is less skew than that of Example 1.
\section{Alignment to a radially symmetric feature}

\noindent In this section we look at radially symmetric features with small length scales. These tend to arise in applications
either in the form of singularities (such as in problems with blow-up \cite{Hou},\cite{BW:06}) or as thin rings, which arise directly as
in singular solutions to NLS \cite{Merle}, or approximately as in the curved fronts we study
in Section 5. We proceed as in the last section in that we study the alignment and
scaling properties of certain exact radially symmetric solutions of the Monge-Ampere equation. We
also study the global geometry and anisotropy of the resulting meshes, including the behaviour close to the domain boundaries.
Initially we look at analytic solutions in radially symmetric domains and then see how these solutions perturb
in domains without radial symmetry.

\subsection{Exact radially symmetric solutions of the Monge-Ampere equation} 
We begin by considering the form of the Monge Ampere equation (\ref{pmaeqa1}) and mesh mapping in the case of radially symmetric solutions in radially symmetric domains. We then consider the nature
of the meshes obtained when the density function approximates a Dirac measure.
Therefore, we let $(x,y)=(R\cos(\Phi),R\sin(\Phi))$ and $(\xi,\eta)=(r\cos(\phi),r\sin(\phi))$, so that $R=\sqrt{x^2+y^2}$, and  $r=\sqrt{\xi^2+\eta^2}$, 
and assume that a circle of radius $r$ in $\Omega_c$ maps to a circle of radius $R$ in $\Omega_p$, under the map $R=R(r)$. 
Furthermore we assume that the boundary of a disc $\Omega_c$ maps to the boundary of a further disc $\Omega_p$, such that $r=R$ at the boundary.
For a density function that is locally radially symmetric about the origin
\[
\rho(\mathbf{x})=\rho(R),
\]
it follows, after some standard manipulations, that there is a radially symmetric solution $P(r)$ of the Monge-Ampere equation satisfying
\[
\Phi=\phi, \quad R = P_r \quad \mbox{and} \quad P_{\xi\xi}P_{\eta\eta}-P_{\xi\eta}^2=\frac{P_{r}P_{rr}}{r}=\frac{R}{r}\frac{dR}{dr}.
\]
The Monge-Ampere equation (\ref{pmaeqa1}) can be written as
\begin{equation}
\rho(R)\frac{R}{r}\frac{dR}{dr}=\theta, \label{radMA}
\end{equation}
where
\begin{equation}
\theta=\frac{\int_{\Omega_p}\rho(R) R \; dR \; d\Phi}{\int_{\Omega_c} r \;  dr \; d\phi.} \label{theta}
\end{equation}
We can now study the local structure of the map defined by this expression.

\vspace{0.2in}

\noindent \begin{lemma} 
(a) The eigenvectors of the Jacobian of the map are
\begin{equation}
\mathbf{e_1}=\frac{1}{r}\left[\begin{array}{c}{\xi}\\{\eta}\end{array}\right], \hspace{1cm}\mathbf{e_2}=\frac{1}{r}\left[\begin{array}{c}{-\eta}\\{\xi}\end{array}\right].\label{eigvecgen}
\end{equation}
The eigenvector $\mathbf{e_1}$ is in the direction of increasing $r$ and $\mathbf{e_2}$ orthogonal to this in the direction of increasing $\Phi$ ($\phi$).

(b) The corresponding eigenvalues are
\begin{equation}
\lambda_1=(r\psi)^{\prime}=\frac{dR}{dr},\hspace{.5cm}\mbox{and}\hspace{.5cm}\lambda_2=\psi=\frac{R}{r} = \theta/(\rho(R)\lambda_1). \label{eigvalgen}
\end{equation}

(c) The skewness measure (\ref{qske}) takes the form
\begin{equation}
Q_s=\frac{1}{2}\left(\frac{rR^{\prime}}{R}+\frac{R}{rR^{\prime}}\right).\label{2dqs}
\end{equation}

\end{lemma}

\vspace{0.1in}

\begin{proof} Letting $\psi:=R(r)/r$, it follows from straightforward manipulations, that the Jacobian matrix (expressed in $(\xi,\eta)$ coordinates) is
\begin{eqnarray}
\mathbf{J}&=&\left[\begin{array}{cc}\psi+\frac{\xi^2\psi^{\prime}}{r}& \frac{\xi\eta\psi^{\prime}}{r}\\ \frac{\xi\eta\psi^{\prime}}{r}& \psi+\frac{\eta^2\psi^{\prime}}{r}\end{array}\right],\nonumber\\
&=& \left[\begin{array}{cc}\frac{\xi}{r}& \frac{\eta}{r}\\ \frac{-\eta}{r}& \frac{\xi}{r}\end{array}\right] \left[\begin{array}{cc}(r\psi)^{\prime}& 0\\ 0& \psi\end{array}\right] \left[\begin{array}{cc}\frac{\xi}{r}& \frac{-\eta}{r}\\ \frac{\eta}{r}& \frac{\xi}{r}\end{array}\right], \nonumber
\end{eqnarray}
and so $J=\psi(r\psi)^{\prime}$, hence the result follows.  \end{proof}

\vspace{0.1in}

\noindent By  (\ref{cbjac2}) such a mesh will be aligned to the metric tensor
\begin{eqnarray}
\mathbf{M}&=& \left[\begin{array}{cc}\frac{\xi}{r}& \frac{\eta}{r}\\ \frac{-\eta}{r}& \frac{\xi}{r}\end{array}\right] \left[\begin{array}{cc}\frac{\theta}{(R^{\prime})^2}& 0\\ 0& \frac{\theta}{\psi^2}\end{array}\right] \left[\begin{array}{cc}\frac{\xi}{r}& \frac{-\eta}{r}\\ \frac{\eta}{r}& \frac{\xi}{r}\end{array}\right]. \label{MT41}
\end{eqnarray}

\vspace{0.1in}

\noindent Integrating (\ref{radMA}) we obtain
\begin{equation}
\int_0^R \rho(R'){R'} \; dR'=\theta\frac{r^2}{2}.\label{radMA2}
\end{equation}
For given $\rho(R) > 0$ this expression implicitly defines a unique monotone increasing function $R(r)$. Once this function is obtained we can explicitly write down the eigenvalues of the Jacobian matrix and thus quantify the skewness of a mesh element.

\vspace{0.1in} 

\noindent NOTE: These results can easily be extended to $n$-dimensional radially symmetric problems. In this case the generalisation of (\ref{radMA2}) is simply
\begin{equation}
{
\int_0^R \rho(R')(R')^{n-1} \; dR'=\theta\frac{r^n}{n}.
}
\label{radMA2nd}
\end{equation}

\noindent We now consider possible forms for the density function $\rho(R)$ which will concentrate the mesh close to certain features. Specifically, consider
\begin{equation}
\rho(R) = 1 + f(R)
\label{rhoform}
\end{equation}
where the function $f(R)$ is an approximation to a Dirac measure with mass
$$\frac{\gamma}{2} \equiv \int_0^{\infty} f(R) R \; dR,$$
which is large close to $R=a$ and small elsewhere. If $a = 0$ this density function will lead to a mesh concentrated at the origin, which will be appropriate for resolving the locally radially symmetric singular solutions encountered when studying blow-up type problems \cite{BW:09},\cite{Hou}. If $a > 0$ this will lead to a mesh concentrated in a thin ring of radius $a$. This will be appropriate for resolving
either a problem with a ring type singularity \cite{Merle} or (as we shall see when we study the Buckley-Leverett equation in Section 5) the resolution of a front in the solution of a PDE which has locally high curvature.

\vspace{0.1in}
\noindent If we substitute the expression (\ref{rhoform}) into (\ref{radMA2}) we can calculate the relation between $r$ and $R$ and hence determine the resulting mesh.
It is immediately evident that there are three separate regions (two if $a = 0$).

\vspace{0.1in}

\noindent 1. An {\em inner region} given by $R \ll a$ for which $\rho(R) \approx 1$ and hence 
\begin{equation}
R \approx \sqrt{\theta} \; r
\label{innerc}
\end{equation}
In this region the mesh is uniform and isotropic and has a scaling factor of $\sqrt{\theta}$. The value of $\theta$ depends upon the boundary conditions and we discuss it
presently.

\vspace{0.1in}

\noindent 2.  A  {\em singular region} in which $R \approx a$ where the mesh is concentrated close to the singular feature. The precise nature of this depends on the function $f(R)$.

\vspace{0.1in} 

\noindent 3. An {\em outer region} given  by  $R \gg a$, away from the singular feature and including the boundary. The form of the mesh in this outer region is given below

\vspace{0.2in}

\begin{theorem} Let $R \gg a$ and assume that $\rho(R)$ takes the form (\ref{rhoform}). Then

(a) The mesh is given by the expression 
\begin{equation}
R \approx \sqrt{\theta r^2 - \gamma}.
\label{bigmesh}
\end{equation}

(b) In this region the eigenvalues of the map are given by
\begin{equation}
{
\lambda_1 \approx \sqrt{\theta} /\sqrt{1 - \gamma/(\theta r^2)}, \quad \lambda_2 = \sqrt{\theta}\sqrt{1 - \gamma/(\theta r^2)}.
}
\label{lameqn}
\end{equation}

(c) The skewness measure is given by
\begin{equation}
{
Q_s \approx \frac{1}{2} \left( \frac{1}{1 - \gamma/(\theta r^2)} + 1 - \gamma/(\theta r^2) \right).
}
\label{radskew}
\end{equation}

\end{theorem}
\vspace{0.2in}

\begin{proof} As $f(R)$ is small if $R \gg a$, it follows from(\ref{radMA2}) that if $R \gg a$ then
$$R^2/2 + \gamma/2 \approx \theta r^2/2.$$
The result (\ref{bigmesh}) then follows. To obtain (\ref{lameqn}) and (\ref{radskew}) note that $\rho(R) \approx 1$ in this region  and apply Lemma 4.1. \end{proof}

\vspace{0.1in}

\noindent NOTE: We can generalise this result to $n$-dimensions in which case we have $R = \left( r^n \theta  - \gamma \right)^{1/n}$, 
so spherical shells are mapped to spherical shells, but cuboids are distorted.

\vspace{0.1in}

\noindent By applying Theorem 4.2 we can deduce the geometrical form of the mesh in this region. We note that whilst the relation (\ref{bigmesh})  
maps circles to circles, it does not map squares to squares. Indeed the image of a large
square centred on the origin will have a leaf-like shape with the sides of the square mapped closer to the origin than the corners. As $r$ and hence $R$ increases, $Q_s$ tends to one, and the mesh becomes asymptotically  isotropic with again a uniform scaling factor of $\sqrt{\theta}$. As $r$ decreases, the value of $Q_s$ in (\ref{radskew}) increases and the mesh becomes more anisotropic. To see this in more detail,
assume that the computational mesh $\tau_c$ is composed of uniform small squares of side $h$ aligned with the coordinate axes. 
A small square lying on a line through the origin parallel to the coordinate axes in the region $r > r_1$ ($R >a$)  will be mapped in turn to a small rectangle of sides $\lambda_1 h$ and $\lambda_2 h$.
In contrast, the squares on lines at an angle of $\pi/4$ or similar through the origin will be mapped into diamonds with interior diagonals of length $\sqrt{2} \lambda_1 h$
and $\sqrt{2} \lambda_2 h$. The smallest angle $\delta$ in such a diamond is given by $ \delta= 2 \arctan(\lambda_1/\lambda_2). $

\vspace{0.1in}

\noindent To examine the role played by the boundaries we consider a map from a circle in a computational domain of radius $r_*$ to one in a physical domain of radius $R_*$. This then determines the value of $\theta$
and in turn the level of ainisotropy at the boundary.

\vspace{0.2in}

\begin{lemma} (a) If the boundary of a disc  of radius $r_*$ is mapped to one of radius $R_*$ and $\alpha_2 \gg 1$ then

\begin{equation}
 \theta = (R_*^2 + \gamma)/r_*^2, \quad \lambda_2 = R_*/r_*, \quad \lambda_1=(R_*^2 + \gamma)/(R_* r_*)
\label{theteqn}
\end{equation}
(b) The anisotropy $Q_{s,*}$ of the mesh at the boundary is given by
\begin{equation}
Q_{s,*} = \frac{1}{2} \left( 1 + \gamma/R_*^2 + \frac{1}{1 + \gamma/R_*^2} \right).
\label{boundskew}
\end{equation}
(c) In the particular case of  $r_* = R_*$, $\lambda_1 = \theta$, $\lambda_2 = 1$ and
\begin{equation}
{
Q_{s,*} = \frac{1}{2} \left( \theta + \frac{1}{\theta} \right).
}
\label{boundthet}
\end{equation}
Hence the skewness of the mesh close to the boundary is small provided that $\theta$ is close to unity.
\end{lemma}

\vspace{0.1in}

\noindent \begin{proof} These results follow immediately from (\ref{bigmesh}) and Lemma 4.1. \end{proof}

\subsection{Explicitly Calculated Meshes for Radially Symmetric Features}

\noindent Now consider a representative density function $\rho(R)$ having the properties of the function
in section 4.1 which is simple enough to allow explicit calculation of the mesh. In particular we take

\begin{equation}
\rho(R)=1 +f(R) \equiv 1+\alpha_1 \; \mathrm{sech}^2(\alpha_2(R^2-a^2)). \label{rho_sech}
\end{equation}
The parameter $\alpha_1 = \max(f(R))$ (assumed large) determines  the  density of the point concentration onto  the feature. A measure of the width of the feature is $1/\alpha_2a$ (assumed small) if $a > 0$ and  
and $1/\sqrt{\alpha_2}$ if $a = 0$.
It is immediate that 
\begin{equation}
\gamma = \alpha_r = \alpha_1/\alpha_2 \quad \mbox{if} \quad a =0, \quad \mbox{and} \quad \gamma = 2 \alpha_r \quad \mbox{if} \quad a > 0.
\label{gamar}
\end{equation}

\vspace{0.1in}

\noindent Using the expression (\ref{radMA2}) it follows that
\[
\int_{0}^{R} (1+\alpha_1 \; \mathrm{sech}^2(\alpha_2((R')^2-a^2)) \; R' \; dR'=\theta\frac{r^2}{2},
\]
and integrating and rearranging both sides we obtain
\begin{equation}
{R^2}+\alpha_r \; \mathrm{tanh}(\alpha_2(R^2-a^2))+ \alpha_r \; \mathrm{tanh}(\alpha_2 a^2)=\theta r^2 =:F(R).\label{FR}
\end{equation}

\vspace{0.1in}

\noindent We will now analyse the solution for the cases of (i) singular (blow-up) solution corresponding to $a = 0$ and (ii) ring solutions
corresponding to $a >  0.$

\subsubsection{Meshes for Singular Solutions}

\noindent When computing solutions with radially symmetric singularities arising over small length scales, 
such as those observed in the calculation of blow-up solutions \cite{BW:06}, \cite{Hou}, we seek meshes which are uniform and isotropic both inside and away from the
singular region, and which have a smooth transition between these regions. Such meshes are obtained by this method. To see this, note from  (\ref{FR})
that for $a = 0$
\begin{equation}
{
R^2 + \alpha_r \; \mathrm{\tanh}(\alpha_2 R^2)  = \theta r^2.
}
\label{firsteqn}
\end{equation}
The singular region, in which the mesh is concentrated, has radius of the order of $R_1 = 1/\sqrt{\alpha_2}$ . For $R \ll  R_1$ it follows from (\ref{firsteqn}) that
\begin{equation}
R \approx  r \sqrt{\theta/(1+\alpha_1)}.
\label{Rbu}
\end{equation}
We observe that $r$ and $R$ are linearly related and hence, as required,  the mesh is uniform and isotropic in this region. The corresponding region in the computational domain 
is then given by $r < r_1$ where
\begin{equation}
\sqrt{\theta} r_1 \approx \sqrt{(1 + \alpha_1)/\alpha_2}.
\label{r1eqn}
\end{equation} 
Note also that going from the computational to the physical domain we see a mesh compression factor of $\sqrt{\theta/(1 + \alpha_1)}$. 

\vspace{0.1in}

\noindent As $R$ increases beyond $1/\sqrt{\alpha_2}$ then $\tanh(\alpha_2 R^2)$ rapidly tends towards unity, and the mesh evolves into the outer region form given in Theorem 4.2. Using Lemma 4.1 we can explicitly calculate the eigenvalues of the transformation, and therefore quantify the level of skewness using the measure $Q_s$, in the regions
close to the singularity and in the far field.  
Specifically,
\[
{\lambda}_2= R/r \approx \left\{\begin{array}{ll}(\theta/(1+\alpha_1))^{1/2}, &\hspace{.2cm}\mbox{for}\hspace{.2cm} R \ll R_1,\\
r^{-1}({\theta r^2-\alpha_r})^{1/2}, &\hspace{.2cm}\mbox{for}\hspace{.2cm} R \gg R_1,\end{array}\right .
\]
and 
$\lambda_1=\theta/(\rho(R) \lambda_2).$
The skewness measure $Q_s$ is then
\begin{equation}
{Q}_s :\approx \left\{\begin{array}{ll}\frac{1}{2}\left(\frac{(1+\alpha_1)}{\rho}+\frac{\rho}{(1+\alpha_1)}\right), &\mbox{for}\hspace{.1cm} R \ll R_1,\\
\frac{1}{2}\left(\frac{r^2\theta}{(r^2\theta-\alpha_r)}+\frac{(r^2\theta-\alpha_r)}{r^2\theta}\right), &\mbox{for}\hspace{.1cm} R \gg R_1.\end{array}\right .\label{QSbu}
\end{equation}
In the singular region $\rho(R) \approx 1+\alpha_1$, so $\lambda_1\approx \lambda_2\approx \sqrt{\theta/(1 + \alpha_1)}$, ${Q}_s\approx 1$, and the mesh is isotropic. If $R \gg R_1$ then $Q_s$ approaches one as $R \to \infty$. Note the value of $Q_s$ here is that given by (\ref{radskew}) with $\gamma=\alpha_r$ and $a=R_1$. As $R$ decreases towards $R_1$ the mesh becomes more anisotropic and $Q_s$, as determined implicitly from (\ref{firsteqn}) takes a maximum value $Q_{s,max}$. 
This maximum value occurs near $R\approx 2/\sqrt{\alpha_2}$ for which
\begin{equation}
Q_{s,max} \approx \frac{1}{2} \left( \frac{4+\alpha_1 \mathrm{tanh}(4)}{4-\alpha_1 (1-\mathrm{tanh}(4))}+\frac{4-\alpha_1 (1-\mathrm{tanh}(4))}{4+\alpha_1 \mathrm{tanh}(4)}\right).\label{qsmax}
\end{equation}

\noindent We now consider two examples of meshes for $r_* = R_* = 1/2$.

\vspace{0.1in}

\noindent If  $\alpha_1=10$, and $\alpha_2=200$ then from (\ref{theteqn}) we have $\theta = 1.2$ and from (\ref{boundthet})  $Q_{s,*} = 61/60$ at the boundary. Hence the mesh has skew elements at the boundary.
The elements of maximum skewness are located just outside the blow-up region and from (\ref{qsmax}) $Q_{s,max}=1.9$. 
The resulting mesh as an image of a $60\times60$ uniform mesh in the computational domain is plotted on the left in Figure \ref{meshapprox_0}, and the structure of the intermediate and outer regions is apparent.
If $\alpha_1=50$, and $\alpha_2=100$, then $\theta = 3$ and $Q_{s,*} = 5/3$. At the boundary
$\lambda_1/\lambda_2 = \theta = 3$, hence the mesh elements will be stretched in the radial direction by a factor of $3$. In the singular region the 
elements are isotropic and the elements in the physical domain will be approximately $\sqrt{3/51} \approx 1/4$ the size of those in the computational domain. The maximum skewness $Q_{s,max}=6.8$ and so the mesh elements will be stretched in the radial direction by a factor of $13$.  The mesh is shown on the right in Figure \ref{meshapprox_0} and shows a much greater degree of skewness. Again we note that much greater skewness would arise  than the example presented here for a larger value of $\theta$.
\begin{figure}[hhhhhhhh!!!!!]
\includegraphics[height=6cm,width=7cm]{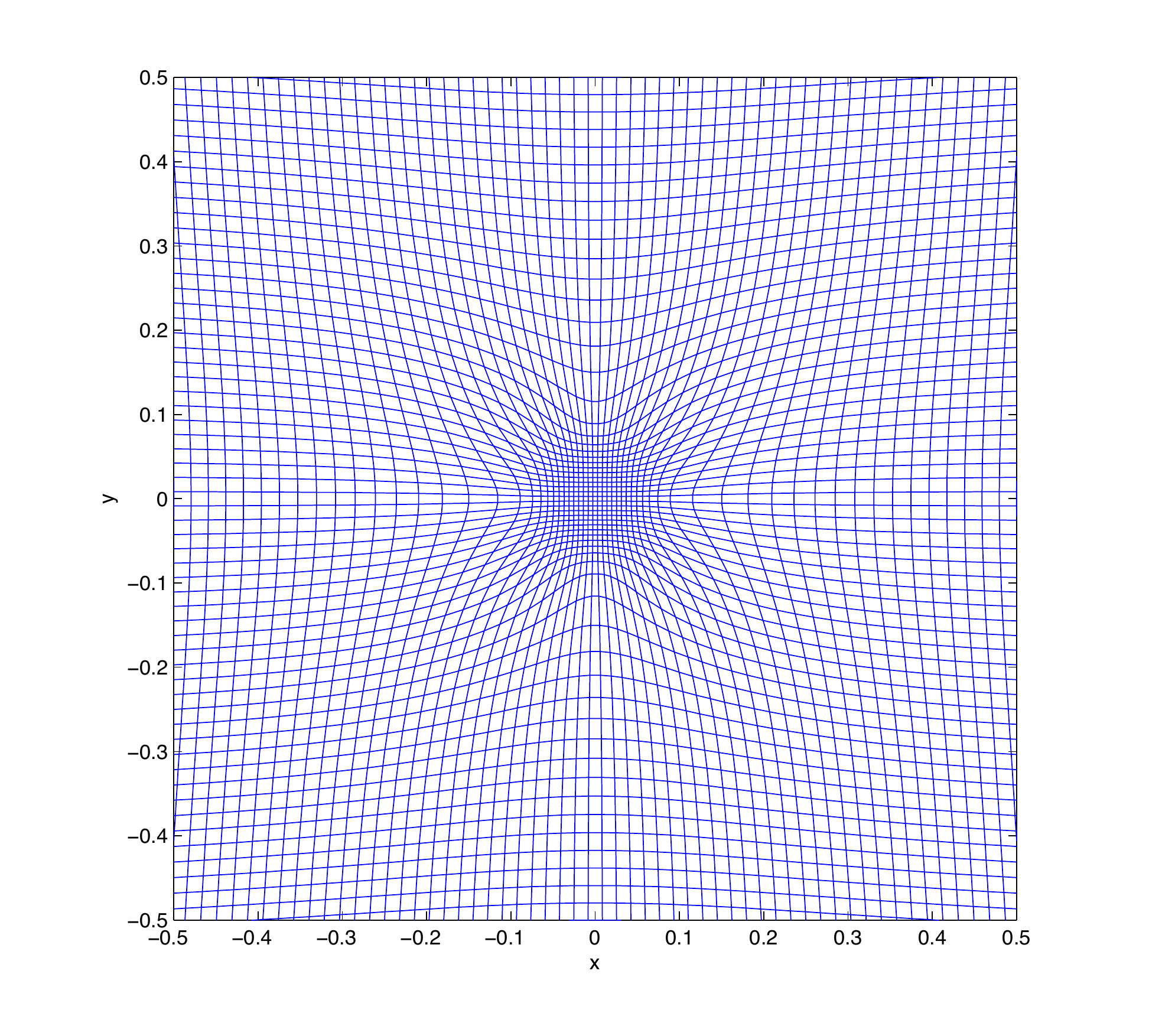} \includegraphics[height=6cm,width=7cm]{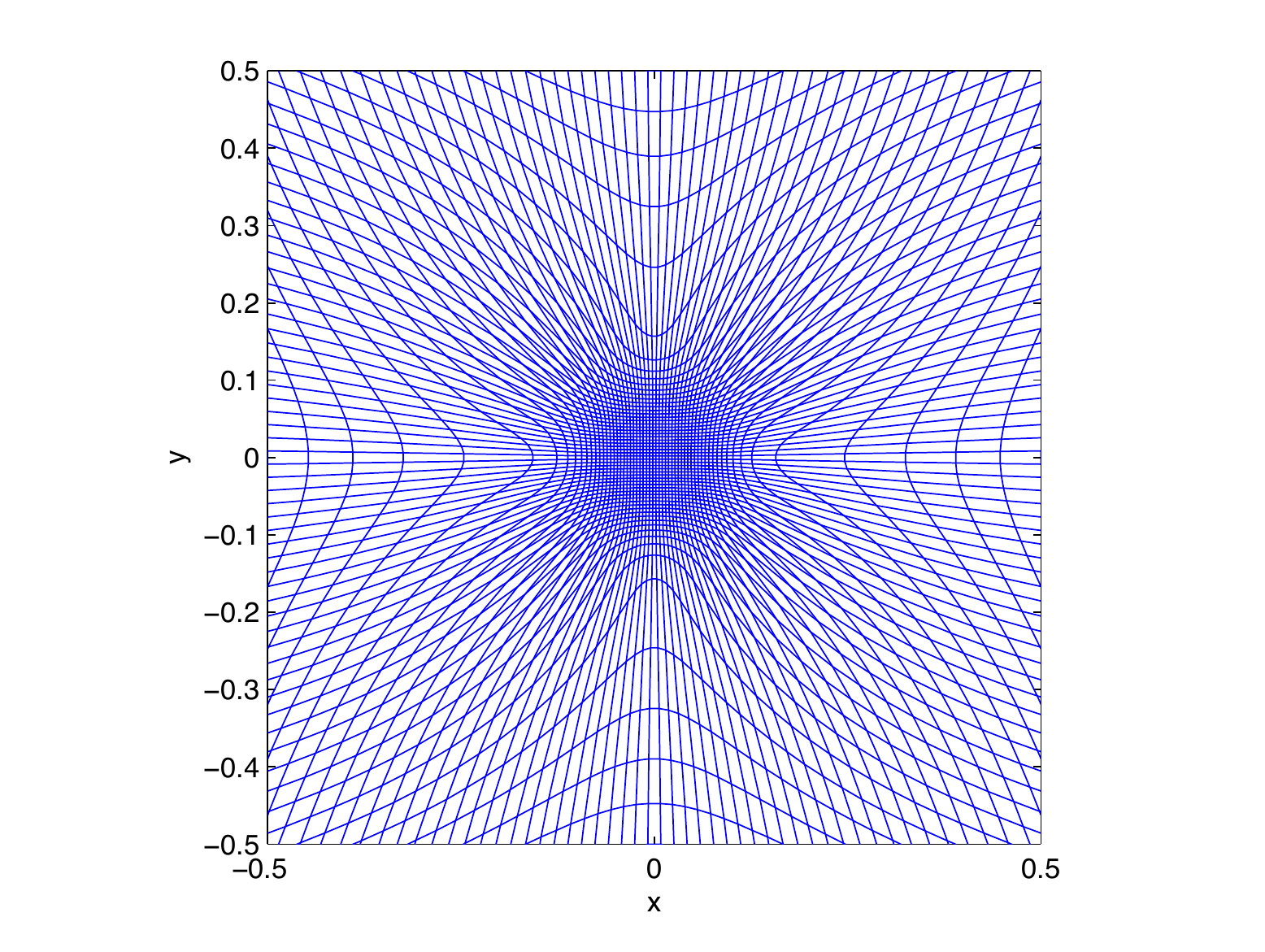}
\caption{ \small The mesh  generated from the image of a regular square mesh ($60\times60$) under the action of a radially symmetric solution of the Monge-Ampere equation for $ a=0$ when $\alpha_1=10$, $\alpha_2=200$, $\theta=1.2$ (left)
and when $\alpha_1=50$, $\alpha_2=100$, $\theta=3$ showing greater skewness (right). The leaf like structure of the mesh in the outer region is apparent in both examples. }
\label{meshapprox_0}
\end{figure}
\vspace{0.1in}
\noindent It is interesting to note these meshes
have the same structure as those generated by the Monge-Ampere method to solve PDE's with blow-up solutions \cite{BW:06}.

\subsubsection{Ring solution}

\noindent We  now consider the case of $ a > 0,\alpha_1 \gg 1, \alpha_2 \gg 1$ so that $\rho \approx 1$ if $|R^2-a^2| > {\cal O}(1/a_2^2)$ and $\rho \approx 1 + \alpha_1$ otherwise, which leads to mesh concentration along a ring. 
For $R \gg a$ the mesh is described by the outer solution considered earlier, with anisotropy at the boundary given by Lemma 4.3.  Similarly, if $R \ll a$ then the mesh is described by the inner region and isotropy with
a scale factor of $\sqrt{\theta}.$
When $a > 0$ the function (\ref{FR}) can be approximated by 
\begin{eqnarray}
{R}\approx\left\{\begin{array}{ll}\sqrt{r^2\theta}, &\hspace{.2cm}\mbox{for}\hspace{.2cm}r \ll  r_1,\\
\sqrt{\frac{r^2\theta-\alpha_r+\alpha_1a^2}{1+\alpha_1}}, &\hspace{.2cm}\mbox{for}\hspace{.2cm} r_1 \ll r \ll r_2,\\
\sqrt{r^2\theta-2\alpha_r}, &\hspace{.2cm}\mbox{for}\hspace{.2cm}r\geq r_2,\end{array}\right .\label{Rrs}
\end{eqnarray}
 where the radii $r_1=\sqrt{\theta^{-1}(a^2-1/\alpha_2)}$, and $r_2=\sqrt{\theta^{-1}(a^2+1/\alpha_2+2\alpha_r)}$ are mapped to $R_1=\sqrt{a^2-1/\alpha_2}$ and $R_2=\sqrt{a^2 + 1/\alpha_2}$, respectively.
Using (\ref{eigvalgen}) and (\ref{Rrs})
\begin{eqnarray}
\lambda_2&\approx&\left\{\begin{array}{ll}{\theta}^{1/2}, &\hspace{.2cm}\mbox{for}\hspace{.2cm}r \ll r_1,\\
((r^2\theta-\alpha_r+\alpha_1a^2)/(r^2(1+\alpha_1)))^{1/2}, &\hspace{.2cm}\mbox{for}\hspace{.2cm}r_1  \ll r \ll r_2,\\
((r^2\theta-2\alpha_r)/r^2)^{1/2}, &\hspace{.2cm}\mbox{for}\hspace{.2cm}r\geq r_2.\end{array}\right .\label{EIG_RS}\\
\lambda_1&\approx&  \theta/(\rho \lambda_2).\nonumber
\end{eqnarray}
Similarly, the level of anisotropy  $Q_s$ can be approximated by
\begin{eqnarray}
{Q}_s:=\left\{\begin{array}{ll}\frac{1}{2}\left(\frac{1}{\rho}+{\rho}\right), &\hspace{.2cm}\mbox{for}\hspace{.2cm}r< r_1,\\
\frac{1}{2}\left(\frac{\theta r^2(1+\alpha_1)}{\rho(r^2\theta-\alpha_r+\alpha_1a^2)}+\frac{\rho(r^2\theta-\alpha_r+\alpha_1a^2)}{\theta r^2(1+\alpha_1)}\right), &\hspace{.2cm}\mbox{for}\hspace{.2cm} r_1 <r< r_2,\\
\frac{1}{2}\left(\frac{\theta r^2}{\rho(r^2\theta-2\alpha_r)}+\frac{\rho(r^2\theta-2\alpha_r)}{\theta r^2}\right), &\hspace{.2cm}\mbox{for}\hspace{.2cm}r> r_2.\end{array}\right .\label{QSring}
\end{eqnarray}
Inside the ring with $R \ll a$, since $\rho\approx 1$, $\lambda_1\approx \lambda_2\approx \sqrt{\theta}$, so ${Q}_s\approx 1$ and the mesh is isotropic. On the ring near $R \approx a$, $\rho\approx 1+\alpha_1$, and the degree of anisotropy depends on the value $\alpha_1a^2-\alpha_r$. The larger this value the more anisotropic the mesh. As $a\rightarrow \sqrt{\alpha_r/(\theta-1)}$ then $\lambda_1\rightarrow \theta/\rho$, and $\lambda_2\rightarrow 1$, hence
$\lambda_1/\lambda_2  \rightarrow\theta/\rho.$
Therefore for a large enough radius of curvature $a$ the anisotropy approaches that of a linear feature, as expected. As the radius of curvature becomes smaller and $a\rightarrow\sqrt{1/\alpha_2}$, $\rho\rightarrow1+\alpha_1$, so 
$\lambda_1/\lambda_2 \rightarrow  r^2\theta/(r^2\theta +\alpha_1a^2-\alpha_r) \rightarrow1,$ and the mesh becomes isotropic. 
\vspace{0.1in}

\noindent For example, if $a=0.25$ and $r_* = R_* = 1/2$, then choosing $\alpha_1=10$, and $\alpha_2=200$, gives $\theta \approx 1.4$ and $Q{s,*} = 1.05$. This results in fairly isotropic elements at the boundary.
Inside the ring the mesh elements are isotropic and ${Q}_s=1$ near the centre of the ring. Along the ring the elements are anisotropic, and $Q_s=3.1$ at $R=0.25$ which is the maximum value. The mesh is shown in Fig. \ref{meshapprox_025}.
\begin{figure}[hhhhhhhh!!!!!]
\includegraphics[height=6cm,width=8cm]{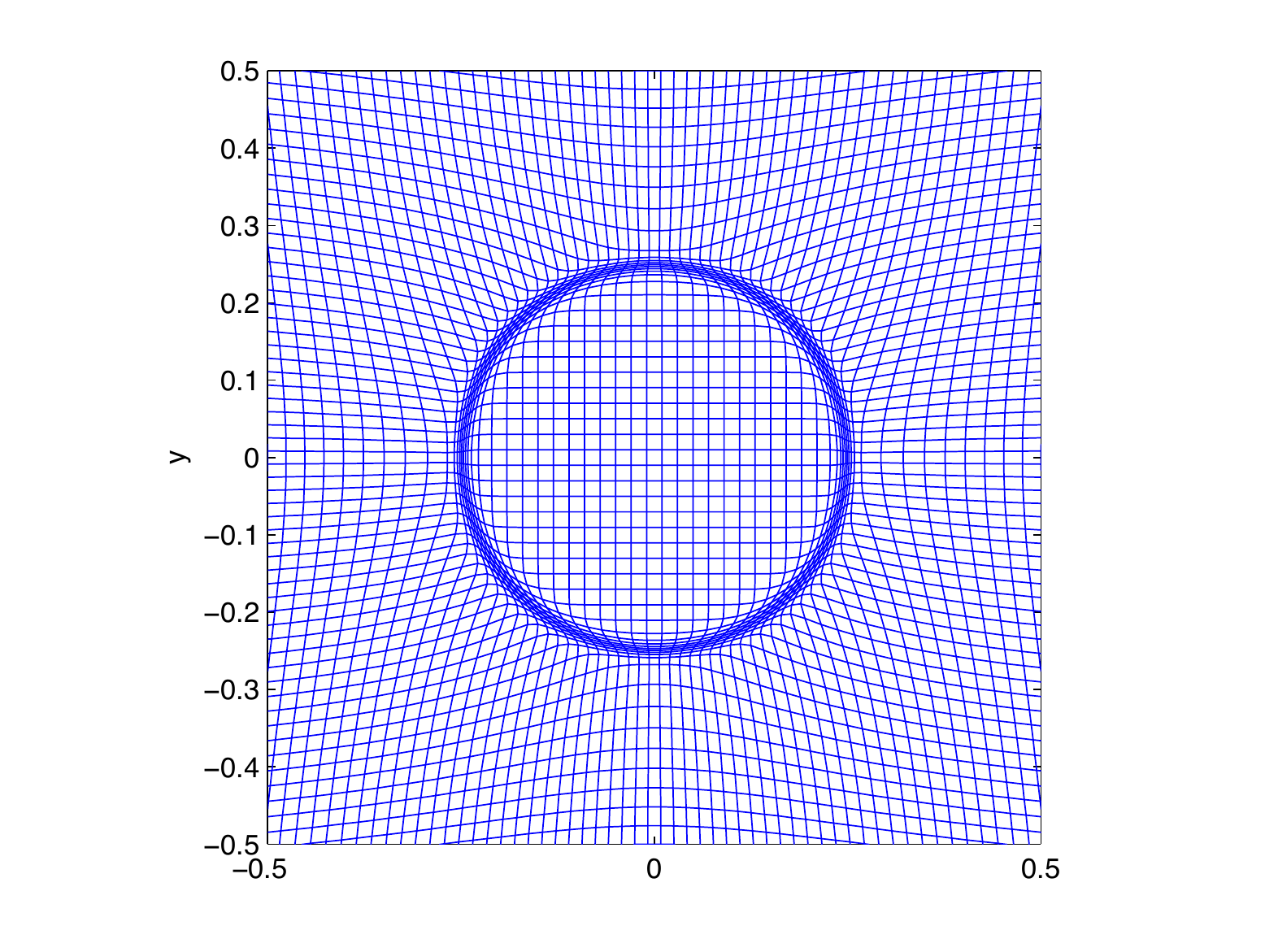} \includegraphics[height=6cm,width=8cm]{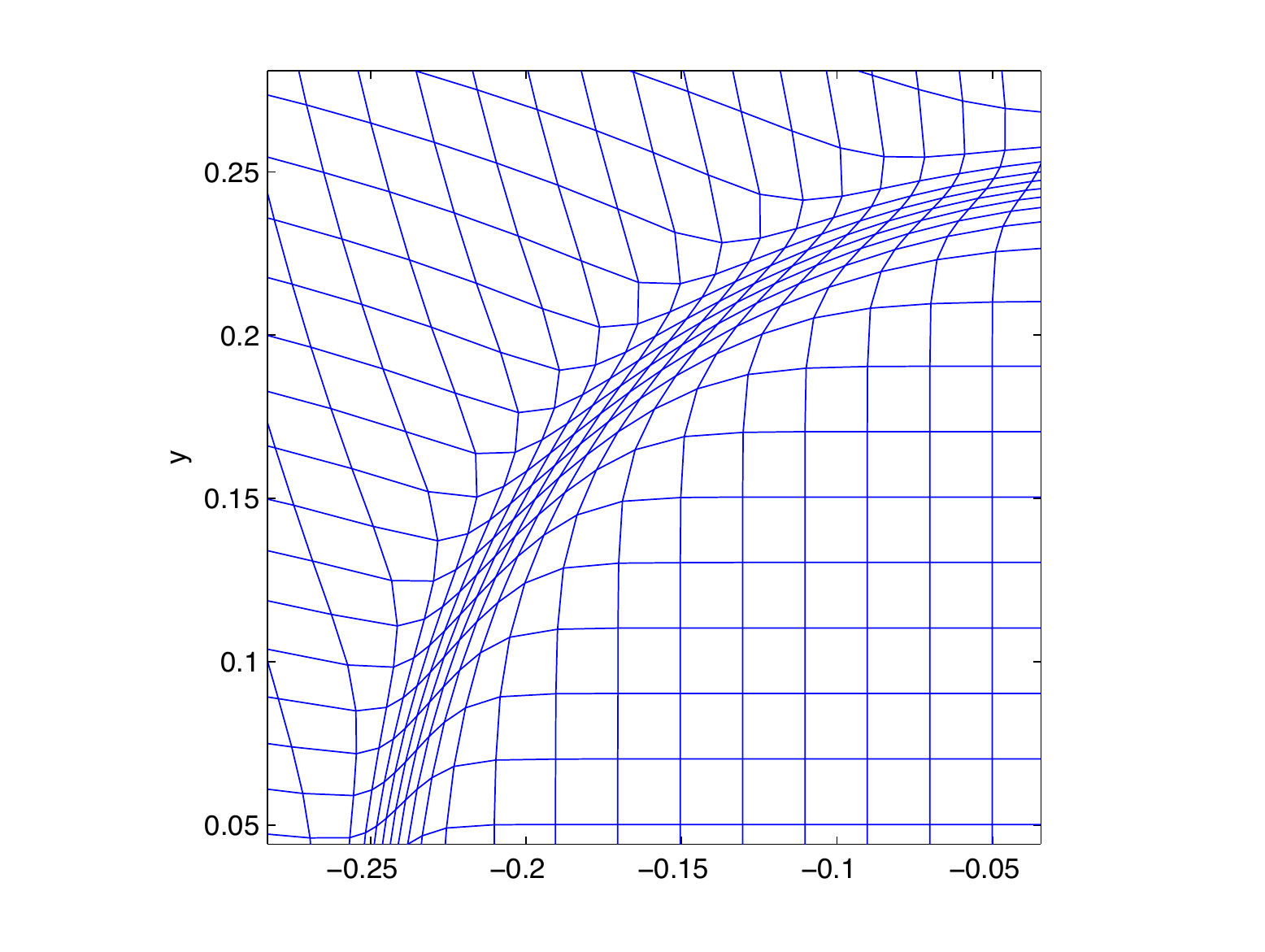}
\caption{ \small A mesh generated from a radially symmetric solution of the Monge-Ampere equation when $\alpha_1=10$, $\alpha_2=200$, $\theta=1.4$ and  $a=0.25$ (left). An enlargement the ring feature (right).}
\label{meshapprox_025}
\end{figure}

\vspace{0.1in}

\noindent If instead $\alpha_1=50$, and $\alpha_2=100$, then $\theta\approx5$ and $Q_{s,*} = 2.6$ so that at the boundary the mesh elements are skew. 
Inside the ring the mesh elements are isotropic with a scale factor of $\sqrt{5}$. However, the maximum value of ${Q}_s \approx 5.1$ does not occur along the ring, as in the previous example, but just outside the ring where elements are stretched in the radial direction. This can be seen in the mesh plot in Fig. \ref{meshapprox3_025}.


\begin{figure}[hhhhhhhh!!!!!]
\includegraphics[height=6cm,width=7cm]{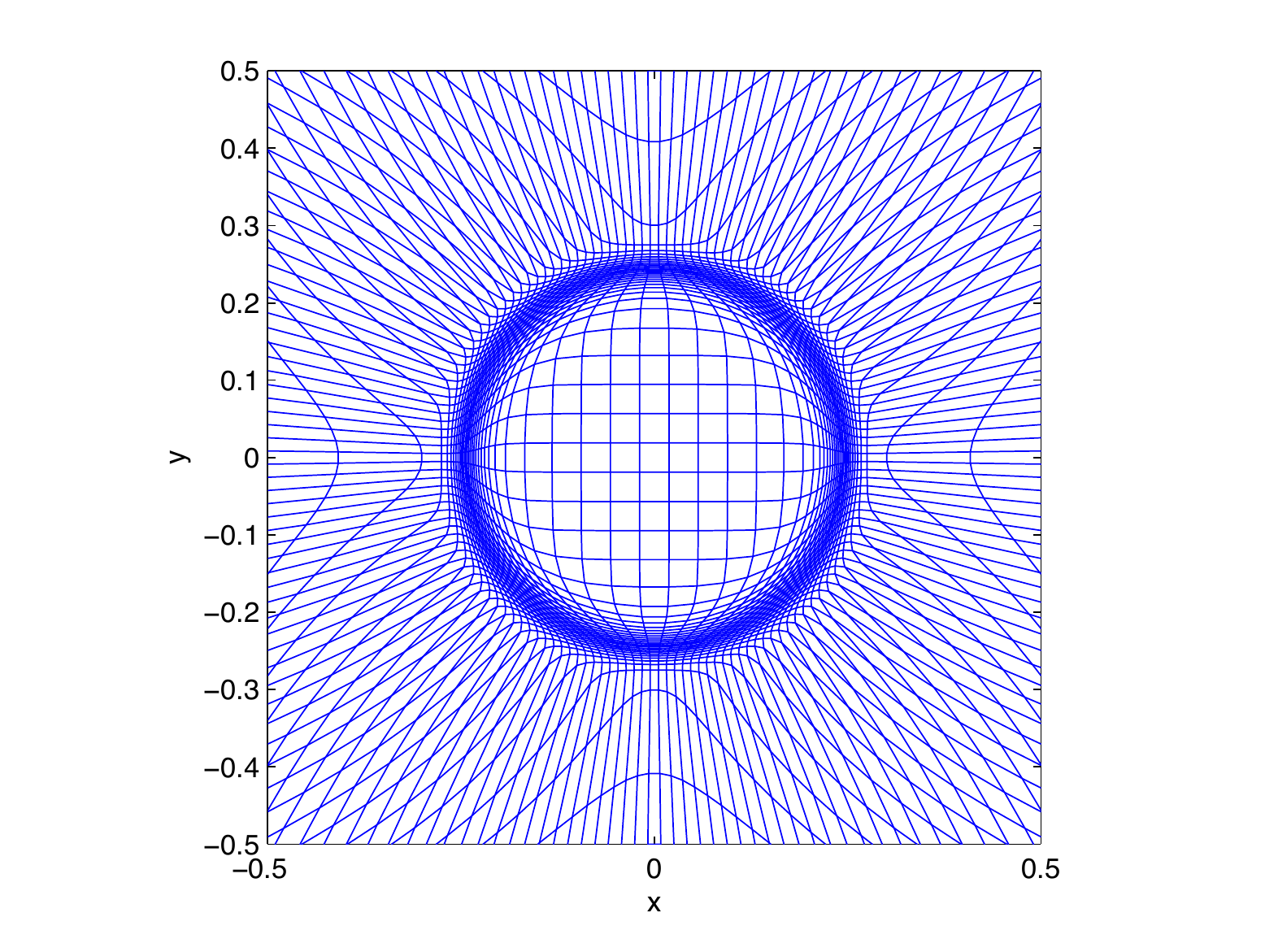} \includegraphics[height=6cm,width=7cm]{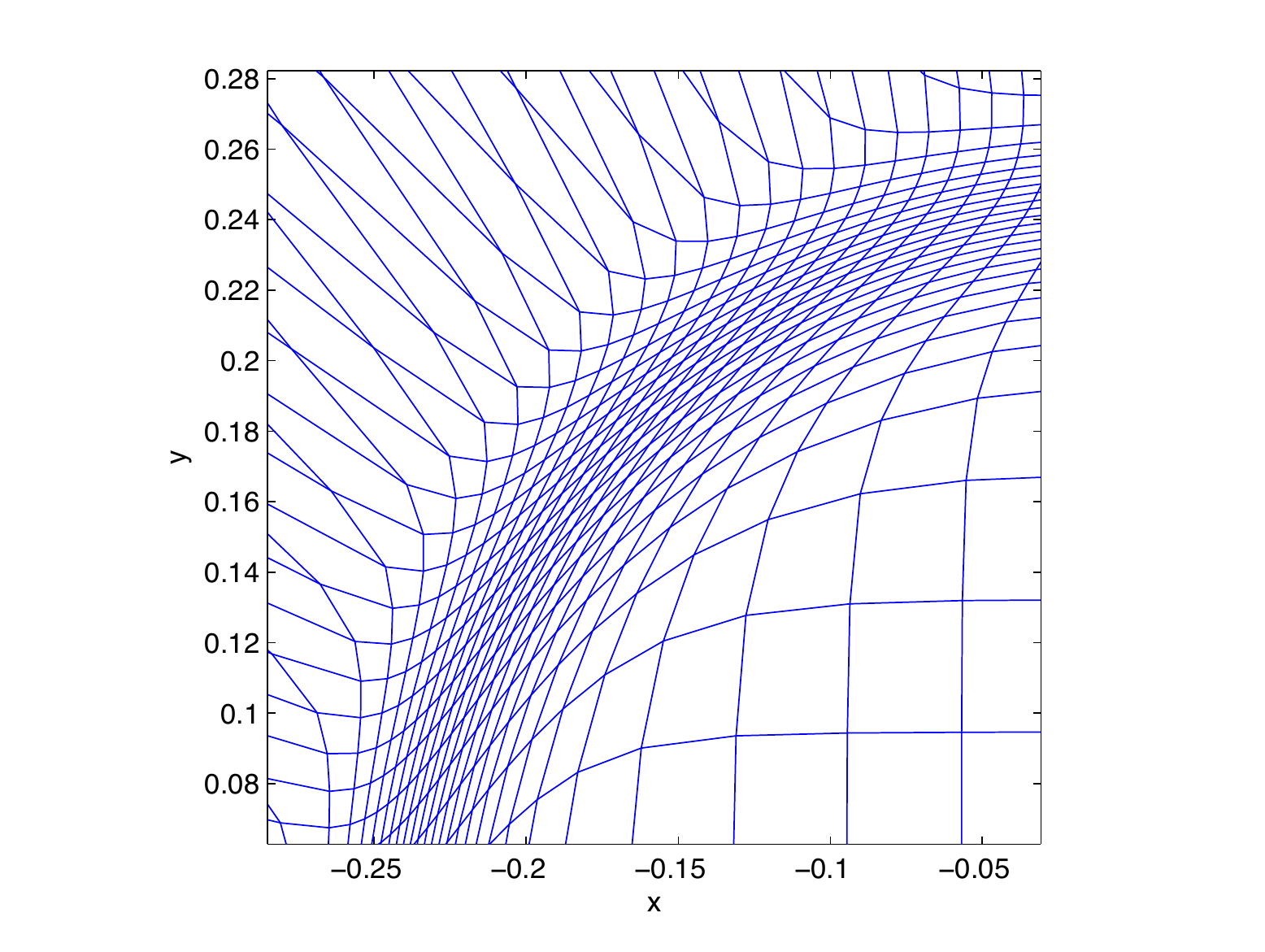}
\caption{ \small A mesh generated from radially symmetric solution of the Monge Ampere equation when $\alpha_1=50$, $\alpha_2=100$, $\theta=5$ and  $a=0.25$ (left). An enlargement of the ring feature (right).}
\label{meshapprox3_025}
\end{figure}

\subsection{Solutions in domains without radial symmetry}
The examples described in the previous section relate to problems in which we can exactly solve the Monge-Amp\`ere equation in a disc, mapping the boundary of a disc to that of another disc.
We now consider problems in more general domains. We note that in the outer region $R \to \sqrt{\theta} \; r$ as $r \to \infty$ so that in the limit square domains
are mapped to square domains. For most such problems the exact solution of the Monge-Ampere equation
is intractable and we must find the solution of  this nonlinear elliptic PDE, together with its associated boundary conditions, numerically. This can either be done directly  \cite{Finn}, \cite{froese}, or by using
a relaxation method  \cite{BW:09}, \cite{walsh}.  In this section we will consider, as before,  the mesh determined for a radially symmetric feature using the density function (\ref{rho_sech}), 
but now for unit square computational and physical domains centred at the origin. It is shown in \cite{BW:06} that the boundary mapping condition is equivalent to imposing Neumann boundary conditions on the solution to the Monge-Ampere equation. This calculation will allow us to assess the impact of boundary conditions on the alignment of the mesh.
In Figure \ref{ringmesh3PMA} (on the left) we see the mesh generated using a numerical solution of the Monge-Ampere equation with Neumann boundary conditions when $\Omega_C=\Omega_P= S \equiv [-0.5, 0.5]^2$, $a=0.25$, $\alpha_1=10$, and $\alpha_2=200$. The skewness measure $\hat{Q}_s$ for this mesh, which is computed numerically using (\ref{qske}), is shown on the right of Fig. \ref{ringmesh3PMA}.
\begin{figure}[hhhhhhhh!!!!!]
 \includegraphics[height=6cm,width=7.5cm]{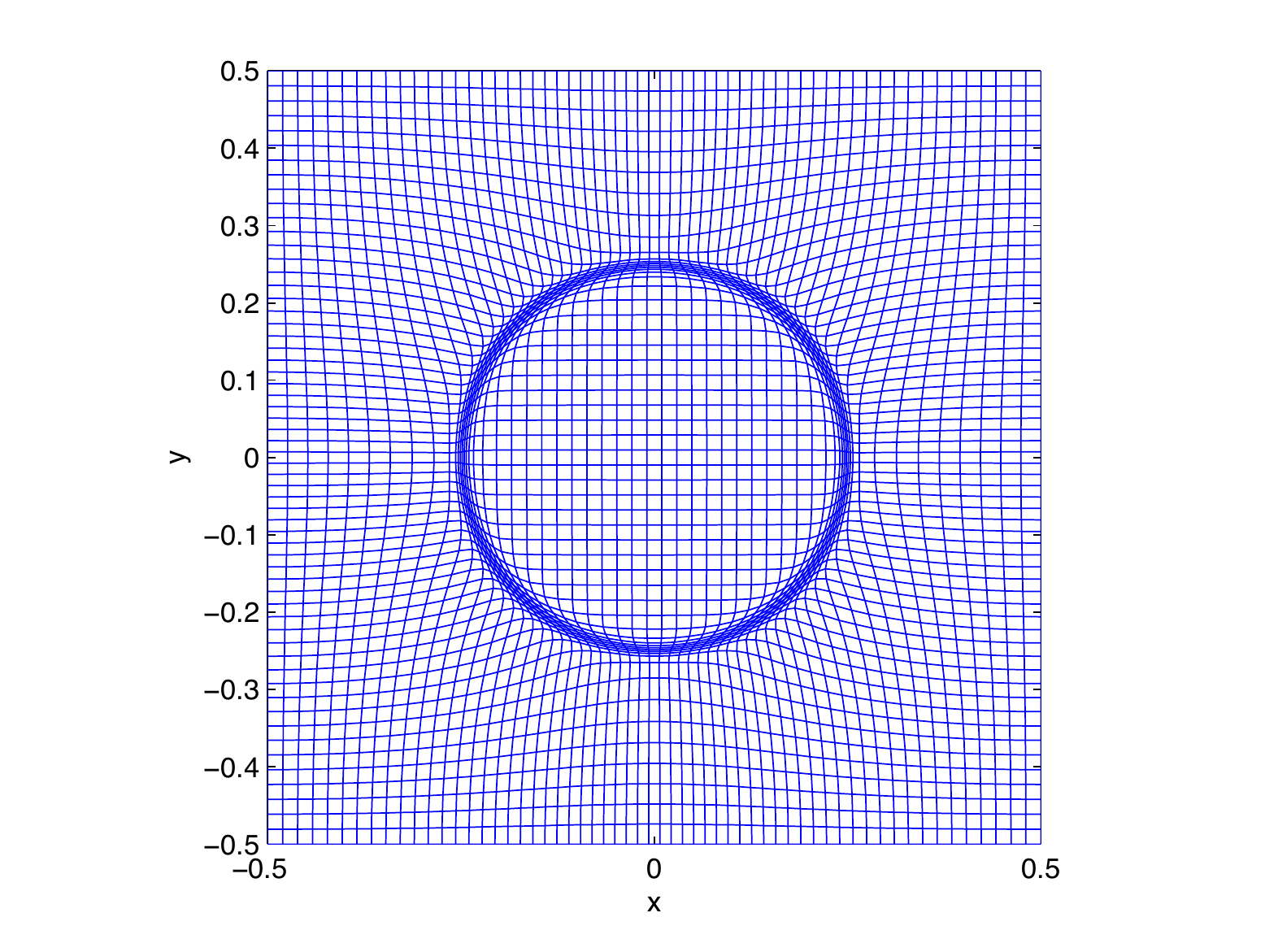} \includegraphics[height=6cm,width=7.5cm]{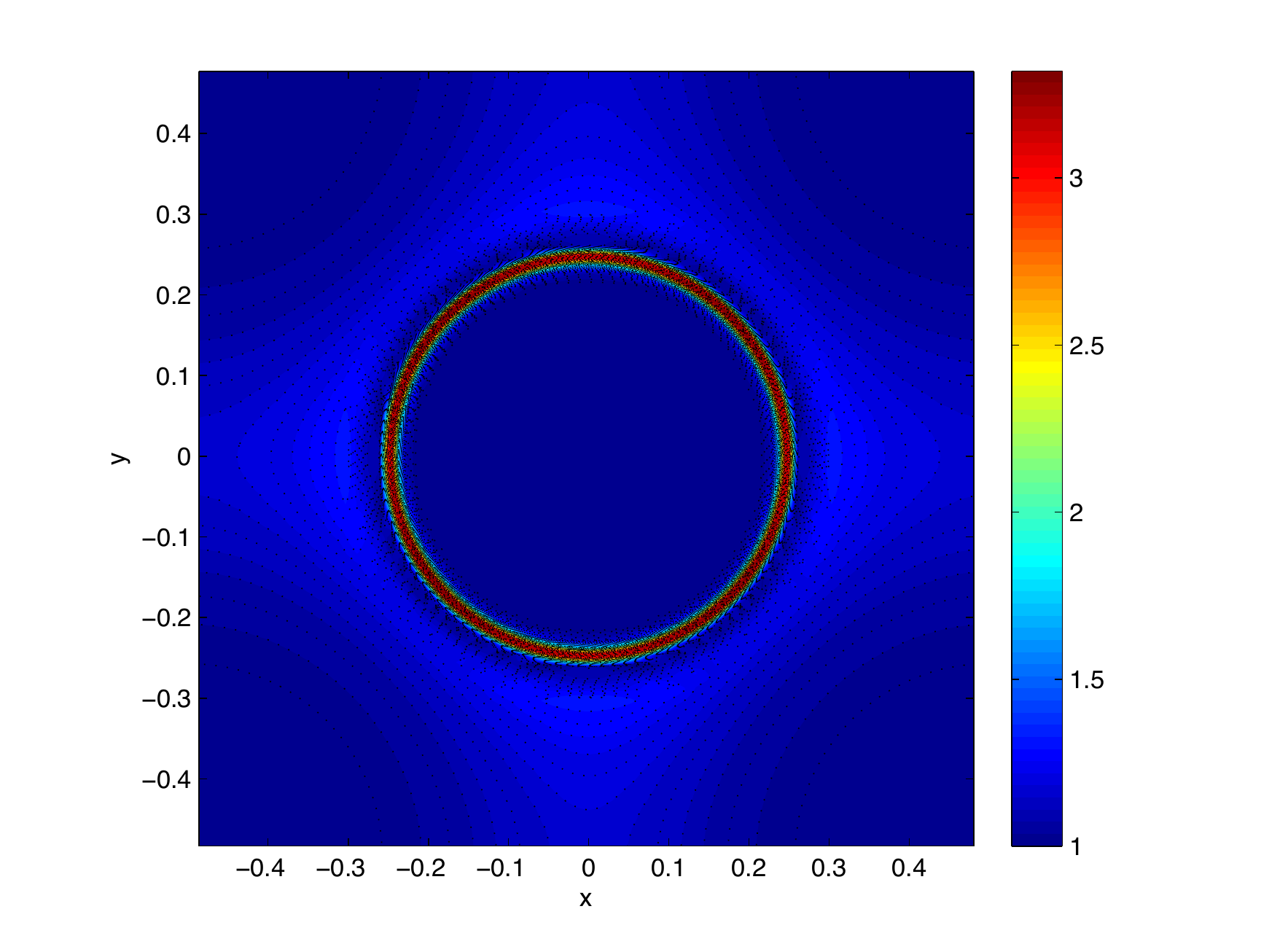}
\caption{ \small The  $(60\times60)$ mesh computed numerically for the density function (\ref{rho_sech}) with  $\alpha_1=10$, $\alpha_2=200$, and $a=0.25$, with boundary $\Omega_C=\Omega_P =[-0.5, 0.5]^2$, (left). The numerically computed skewness measure $\hat{Q}_s$ (right).}
\label{ringmesh3PMA}
\end{figure}
A comparison of $\hat{Q}_s$ with the skewness measure ${Q}_s$ for the radially symmetric solution in (\ref{QSring}), reveals that the effects of the square geometry on the skewness of the mesh are negligible. 
The skewness is almost radially symmetric for the mesh generated in the unit square, although the skewness of elements that lie along the axis $y=0$ and those that lie along $y=x$ differ slightly. The values of ${Q}_s$ at $R=0$, $R=a$, and $R=1/2$, are $1$, $3.1$, and $1.05$ respectively. The values of $\hat{Q}_s$ at $(0,0)$, $(a,0)$, $(1/2,0)$, are 1, 3.1, 1.2, where as at  $(0,0)$, $(a/\sqrt{2},a/\sqrt{2})$, $(1/2,1/2)$ they are 1, 3.3, and 1. 
In Figure \ref{blowupPMA} (on the left) we see the mesh generated numerically when $\Omega_C=\Omega_P=S$, $a=0$, $\alpha_1=50$,  $\alpha_2=100$, and (on the right) the numerically computed skewness measure $\hat{Q}_s$. As in the previous section, this mesh is much more skew outside the blow-up region.
\begin{figure}[hhhhhhhh!!!!!]
 \includegraphics[height=6cm,width=7.5cm]{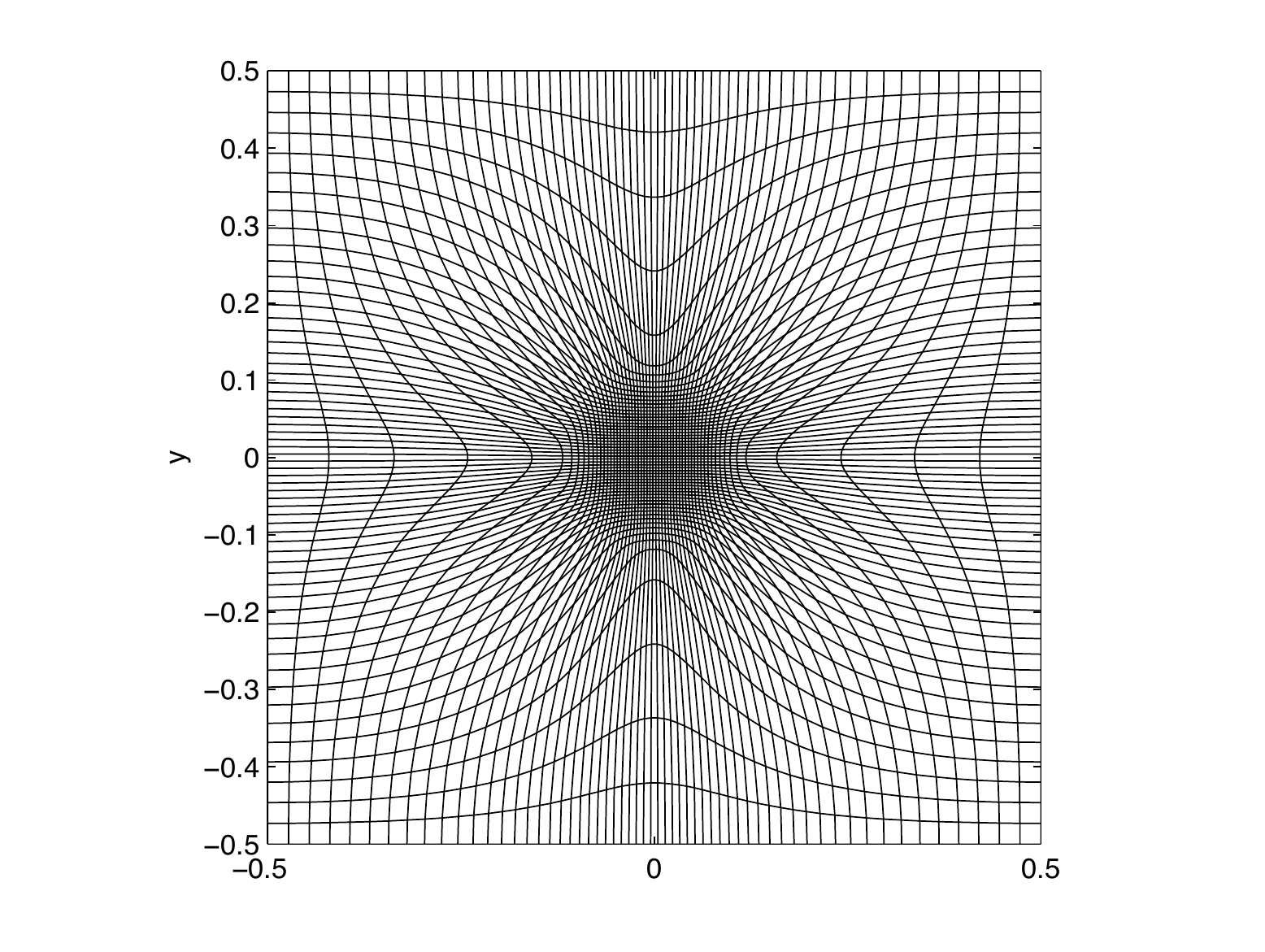} \includegraphics[height=6cm,width=7.5cm]{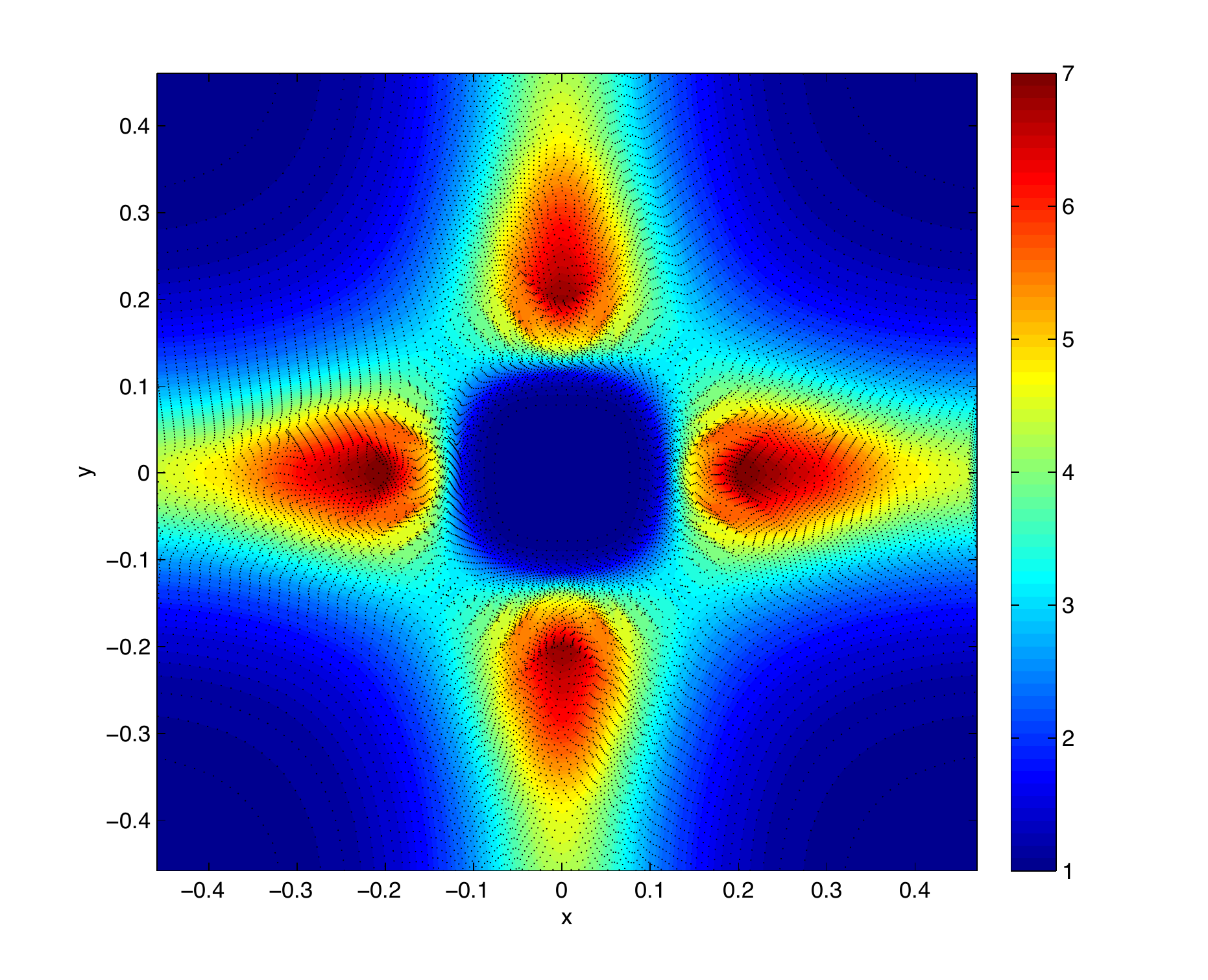}
\caption{ \small  Numerically computed  $(60\times60)$ mesh in $S$ for the density function (\ref{rho_sech}), with $\alpha_1=50$, $\alpha_2=100$, and $a=0$ (left). The numerically computed skewness measure $\hat{Q}_s$ (right).}
\label{blowupPMA}
\end{figure}
However, in this case the skewness is clearly not radially symmetric, and we see significant effects of the square geometry as we approach the boundary.
Elements that lie along the axis $y=0$ and those that lie along $y=x$ do not have exactly the same measures of skewness. Recall that for the radially symmetric solution the values of ${Q}_s$ at $R=0$, and $R=1/2$ are $1$ and $5/3$, and the maximum value of $Q_s$ is $6.5$, and occurs behind the region of blowup. For the numerically computed mesh, along the axis $y=0$, the value of $\hat{Q}_s$ at $(0,0)$, and $(1/2,0)$, is $1$ and $4.4$.  Therefore, at the boundary, the skewness is more than double that of the mesh generated from the radially symmetric solution. The maximum skewness $\hat{Q}_s=7.1$ is also slightly greater than the radially symmetric case, although we note that it occurs at the exact same point just behind the region of blow-up. Along the axis $y=x$ the elements in the numerically computed mesh are not as stretched in the radial direction as in the radially symmetric case. The maximum value of $\hat{Q}_s$ is $3$ and occurs just behind the singular region.  At the boundary the value is only 1.2. We also obtain similar results for the ring case in the region outside the ring. In particular, when $\theta$ exceeds 1 the effects of the geometry become more significant,  and the larger the value of $\theta$ the more skew the elements are near the boundary. If $\theta$ is much larger than 1 the elements of greatest skewness occur just outside of the ring and not at the boundary. However, inside the ring and more importantly along the ring the values of $\hat{Q}_s$ and $Q_s$ do not differ significantly, hence the geometry of the mesh has very little impact on the degree of anisotropy in these regions.

%

\section{Examples of mesh alignment to more general features}

\noindent The exact calculations presented in the previous two sections have looked at features with simple geometries, while in practical calculations the mesh can have a much more complex geometry. In this section we will consider two examples of such, and consider the geometry of the meshes computed numerically by solving (MA) for an appropriate density function $\rho$.
The first example  has a prescribed (scalar) density $\rho$ and the second has $\rho$ given in terms of the evolving solution of a PDE which is known to develop complex features on small length scales. In both cases the features have 
certain sections which are similar to the linear features of Section 3 and we shall see similar alignment of the meshes close to them. Similarly, they also have features with curvature, in which case the results of Section 4 can be used to predict the (local) geometry of the mesh.

\subsection{Example 1: A prescribed monitor function}
Consider the density function
\begin{equation}
\rho=1+\alpha_1 \; \mathrm{sech}^2(\alpha_2\vert \Psi\vert), \hspace{1cm}\Psi=y-0.2 \; \sin(2\pi (x+0.5)).\label{sin}
\end{equation}
which describes a sinusoidal feature of thin cross-section. We will consider both the local and the global geometry of the mesh that results when solving the Monge-Ampere equation in a square domain with Neumann boundary conditions in the y-direction and periodic boundary conditions in the x-direction. Such boundary conditions are appropriate for the solution of periodic waves such as (\ref{sin}),  and arise naturally in many meteorological applications \cite{walsh}.
In Fig.~\ref{mesh6} the numerically calculated mesh with this density function, for $\alpha_1=20$, $\alpha_2=100$, $\theta=1.2$ with the above boundary conditions, and the corresponding ellipses for the Jacobian $\mathbf{J}$, are shown on the right.  
\begin{figure}[hhhhhhhh!!!!!]
\includegraphics[height=5.2cm,width=7cm]{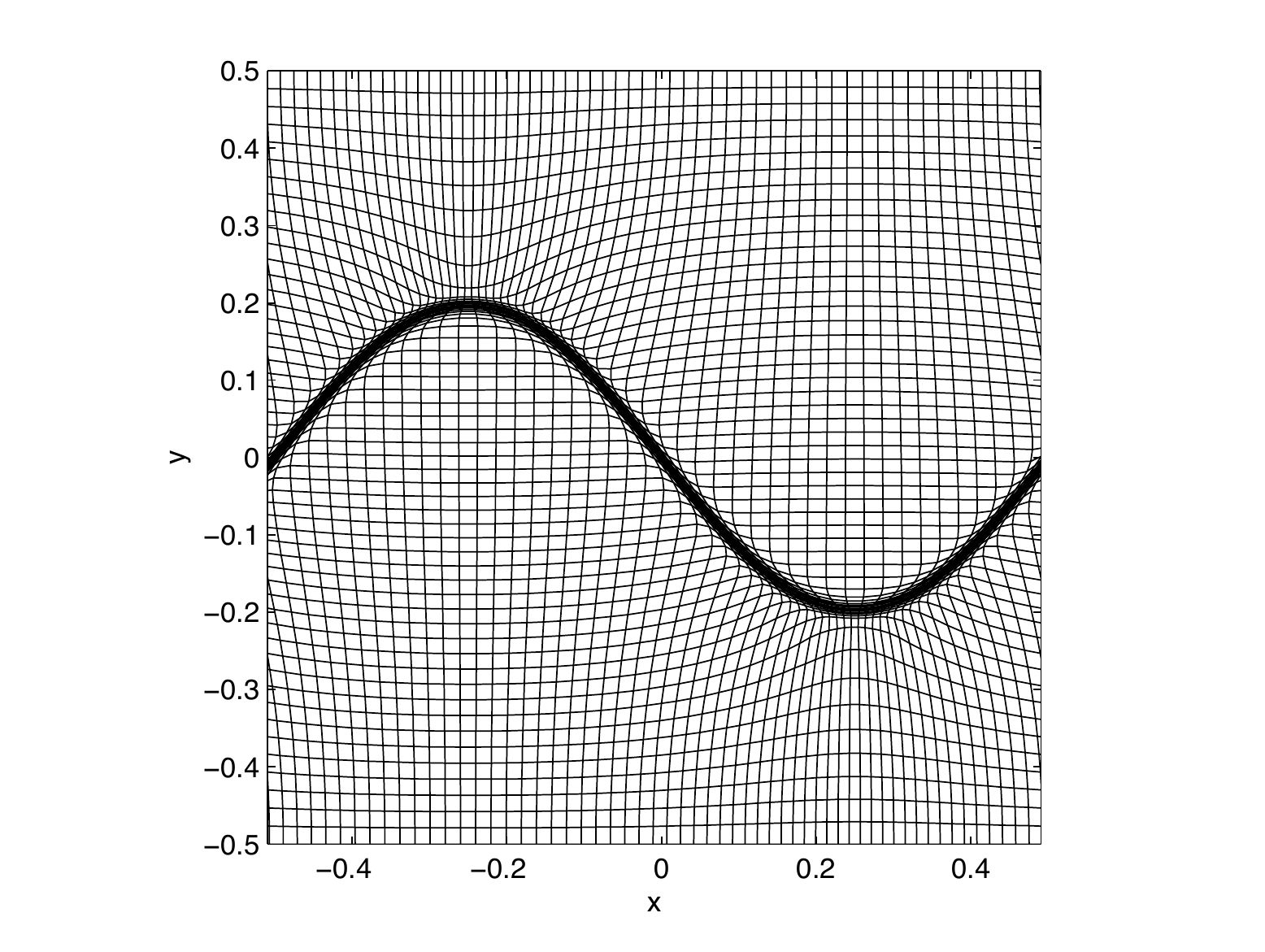}
\includegraphics[height=5.2cm,width=6cm]{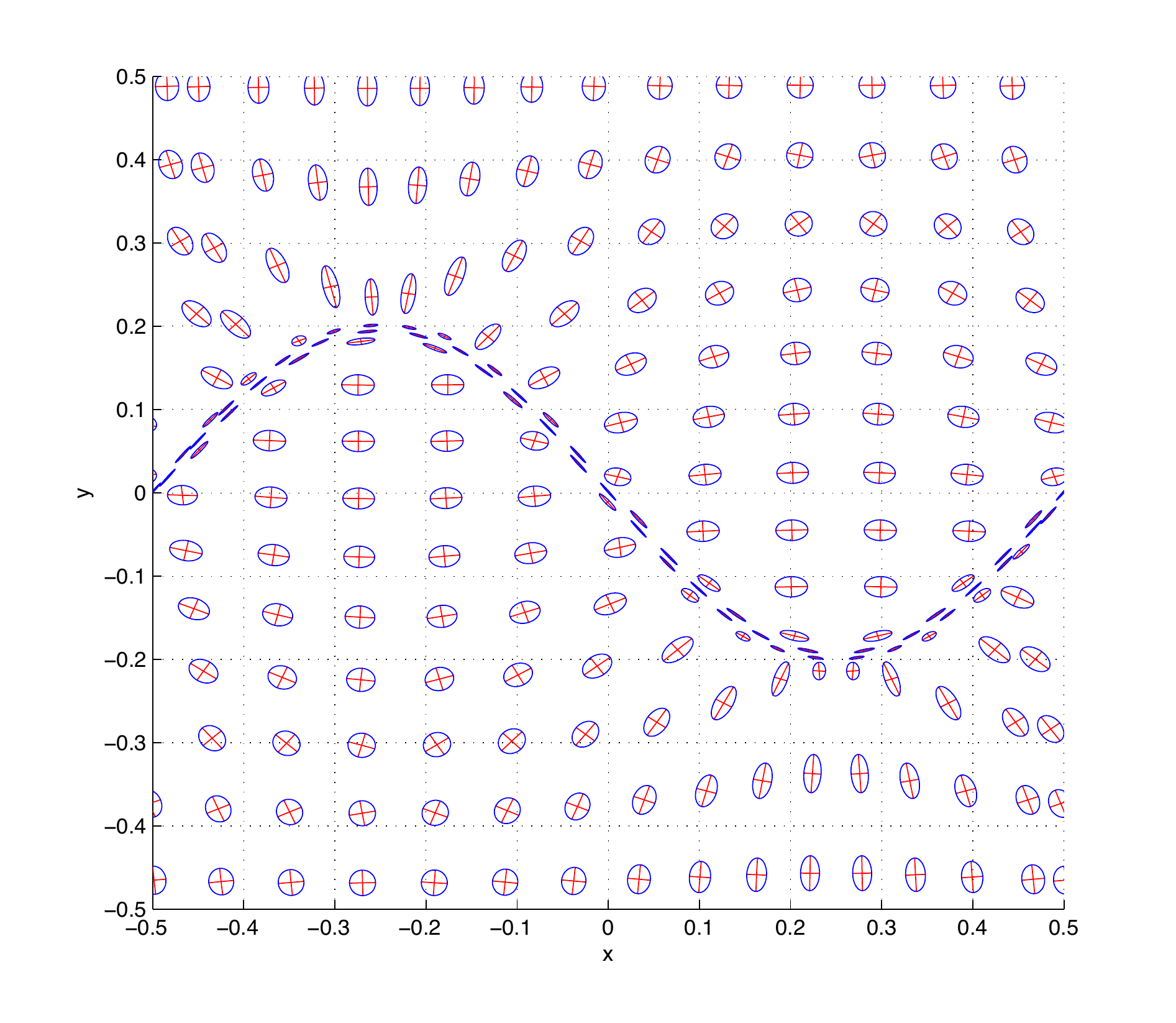}\\
\caption{ \small The numerically computed mesh (60 $\times$ 60) generated for the density function (\ref{sin}) with $\alpha_1=20$, $\alpha_2=100$, $\theta=1.4$, with Neumann boundary conditions n the y-direction and periodic in the x-direction (left), and the circumscribed ellipses of the Jacobian $\mathbf{J}$ (right).}
\label{mesh6}
\end{figure}
It would appear that the eigenvectors of $\mathbf{J}$  are orthogonal and tangential to the curve defined as the set for which $\Psi(\mathbf {x})=0$. Given that $\rho$ is constant along this curve, it is reasonable to assume there will be no movement of the mesh in that direction, so the eigenvalue corresponding to the tangential eigenvector is estimated to be 1, implying the eigenvalue in the orthogonal direction is $\theta/\rho$. 
The symmetric matrix $\mathbf{\tilde{J}}$ corresponds to a metric tensor $\mathbf {\tilde{M}}$ with eigenvalues and eigenvectors given by
$$ \tilde{\mu}_1=\rho^2/\theta, \hspace{.5cm} \tilde{\mu}_2=\theta, \hspace{.5cm} \mbox{and}\hspace{.5cm} \mathbf{\tilde{e}}_1={\nabla \Psi}/{\| \nabla \Psi \|}.$$  
 Notice that these eigenvalues correspond to those derived in Section 3 for a single  linear feature where $\Psi=\mathbf{x}\cdot\mathbf{e_1}-c$.  This is a very good approximation 
in the regions along the sinusoidal feature that are close to linear, where we observe good alignment to the feature (see Fig.~\ref{metrics_example6b} (right)). Furthermore, the mesh is close to being uniform away from the feature.
However, in regions where the feature has more curvature, the mesh elements are less anisotropic (see  Fig.~\ref{mesh6}). Interestingly, the eigenvalues of $\mathbf {\tilde{M}}$ are not a good approximation in this region but the eigenvectors are. In fact we observe that the eigenvectors are approximately tangential and orthogonal to the shock along the entire curve $\Psi({\mathbf x}) = 0$.
For these regions of maximum curvature of the feature, we instead approximate the eigenvalues of the metric tensor by using the radially symmetric solution of the Monge-Ampere equation studied in Section 4. Specifically we assume that in the region $x_1<x<x_2,$ \& $y<y_1$, $\rho$ can be well approximated as part of a radially symmetric feature with density function
\[
\hat{\rho_1}=1+\alpha_1 \; \mathrm{sech}^2(\alpha_2\vert \Psi_1\vert)\nonumber, \hspace{1cm}\Psi_1=\hat{R_1}^2-{a}^2,\nonumber
\]
and similarly in the region $-x_2<x<-x_1,$ \& $y>y_1$, by
\[
\hat{\rho}=1+\alpha_1 \; \mathrm{sech}^2(\alpha_2\vert \Psi_2\vert)\nonumber, \hspace{1cm}\Psi_2=\hat{R_2}^2-{a}^2,\nonumber
\]
where $\hat{R_1}=\sqrt{(x+0.25)^2+(y+{a}-0.2)^2}$ and $\hat{R_2}=\sqrt{(x-0.25)^2+(y-{a}+0.2)^2}$. The radius $a$ of the radially symmetric feature is estimated by taking the average radius of curvature along a section of $\Psi$. We can then approximate the eigenvalues tangential and orthogonal to $\Psi$ as in Section 1.1.4. 
Note that we calculate $\theta$ using the integral of the original density function $\rho$ over the domain, rather than $\hat{\rho}$. The numerical Jacobian $\mathbf{J}$, when solving PMA using doubly periodic boundary conditions, is compared to an approximation of the Jacobian $\mathbf{\hat{J}}$ using the eigenvalues from the radially symmetric solution (see Fig. {\ref{metrics_example6b}}(left)), for $\alpha_1=20$, $\alpha_2=100$, $\theta=1.4$, $a=0.25$, and $x_1=0.18$, $x_2=0.32$, and $y_1=0.18$. The circumscribed ellipses for a number of elements near a region of high curvature are shown
together with their semi-axes, which are depicted in red. The semi-axes of the the ellipses associated with $\mathbf{\tilde{J}}$ are shown in black.
\begin{figure}[hhhhhhhhHHHHHHHHHHHHHHHHHHHHHHHHHHHHHHHHHHHHHHHHHHHHHHHHHHHHHHHHHHHHHHHHHH!!!!!]
\includegraphics[height=5.2cm,width=6cm]{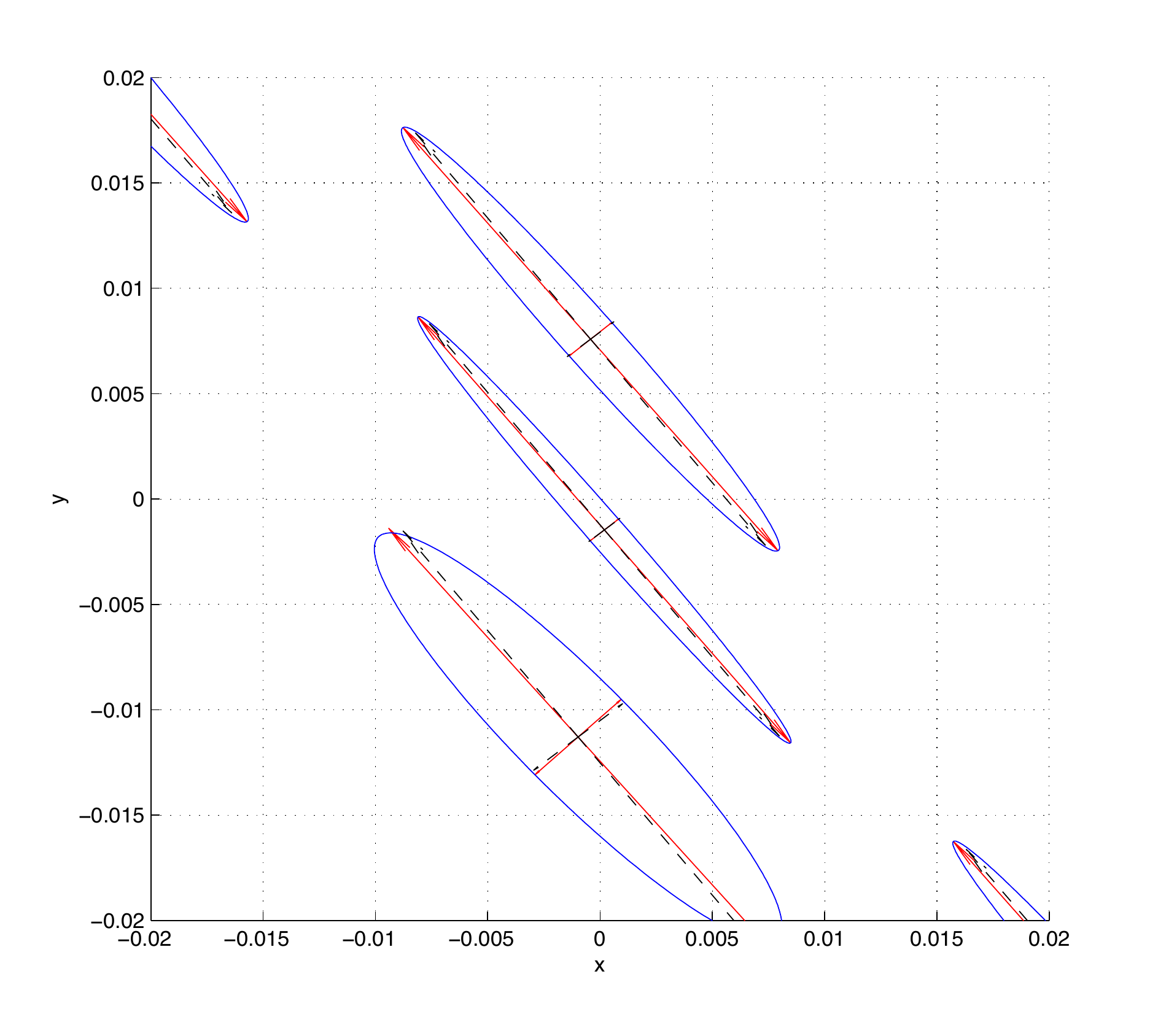}\includegraphics[height=5.2cm,width=6cm]{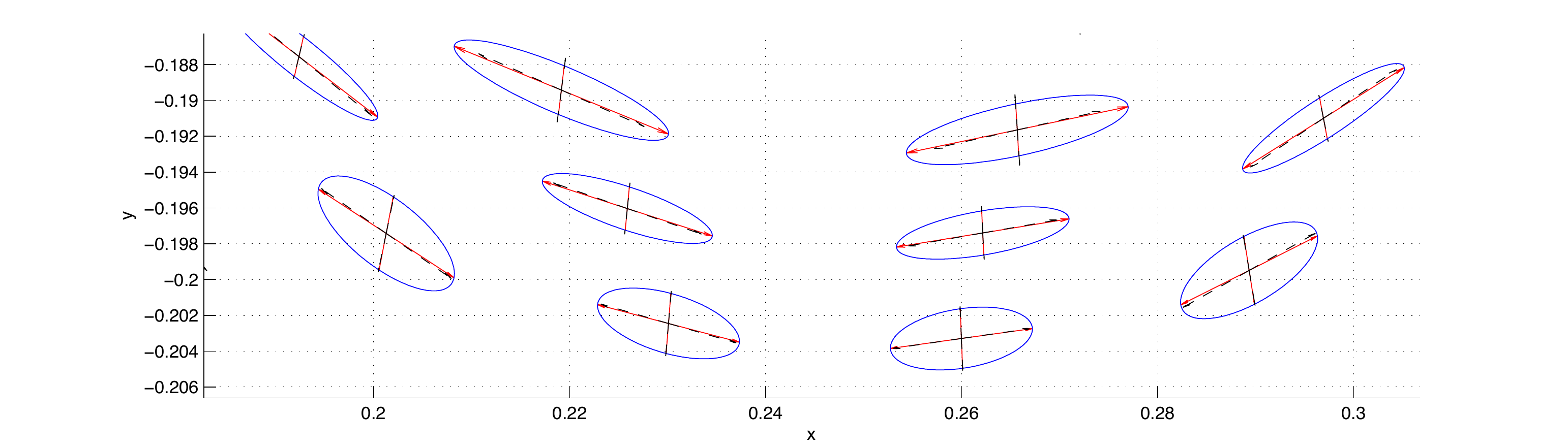}\\
\caption{The eigenplot for $\mathbf{J}$ (red) and $\mathbf{\tilde{J}}$ (black) along the feature, where $\Psi$ is approximately linear (left).The eigenplot for $\mathbf{J}$ (red) and $\mathbf{\hat{J}}$ (black) along the feature  in a region of high curvature (right).}
\label{metrics_example6b}
\end{figure}
If we instead choose $\alpha_1=50$, $\alpha_2=50$, such that $\theta=3$, then $\mathbf{J}$ is well approximated by $\mathbf{\tilde{J}}$ and $\mathbf{\hat{J}}$ in the linear regions and regions of high curvature, respectively, along $\Psi$. 
The mesh and  $\mathbf{J}$ are shown in Fig. \ref{sinemesh}.
\begin{figure}[hhhhhhhh!!!!!]
\begin{center}\includegraphics[height=6cm,width=7.5cm]{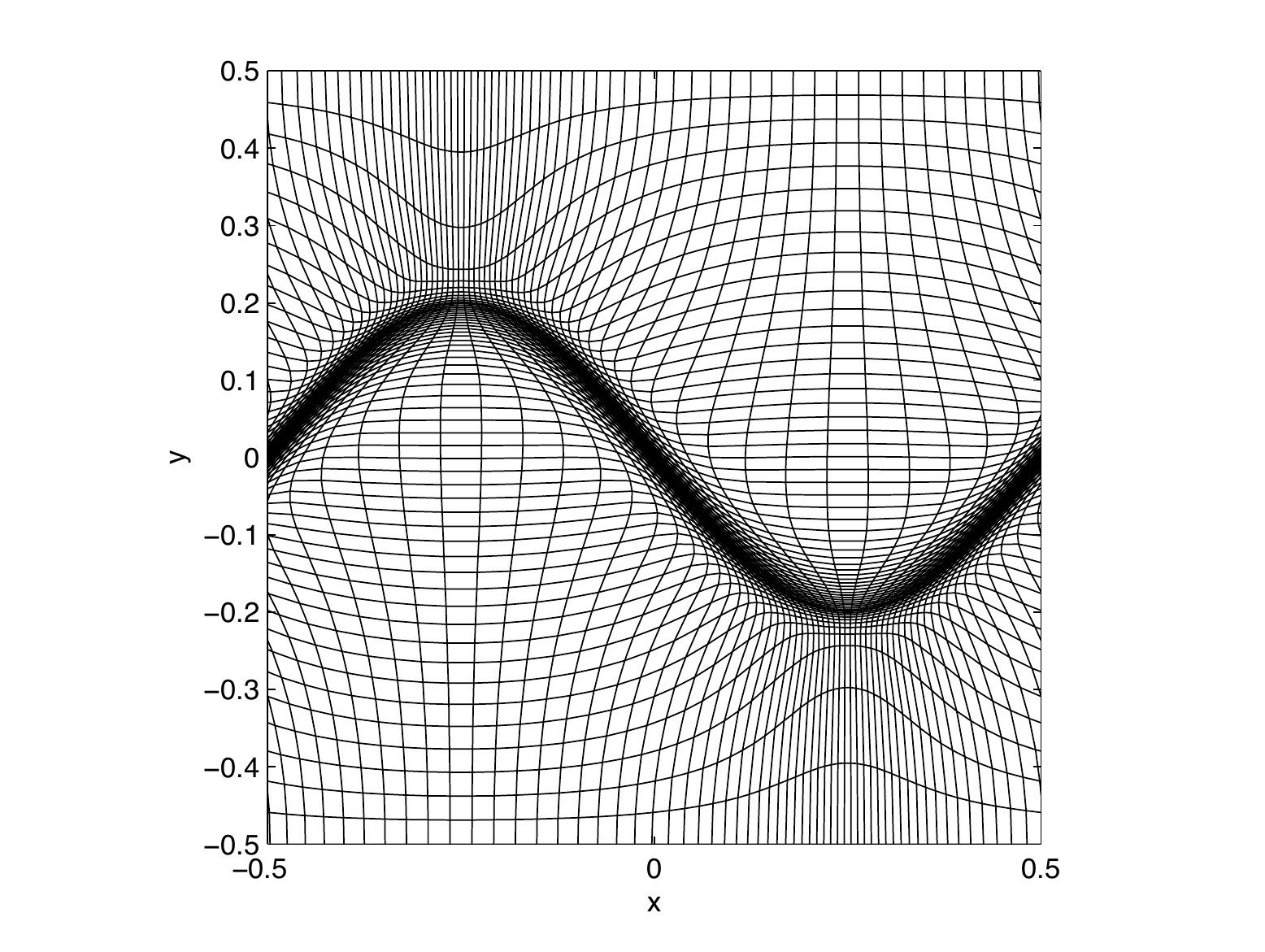}\includegraphics[height=6cm,width=6.5cm]{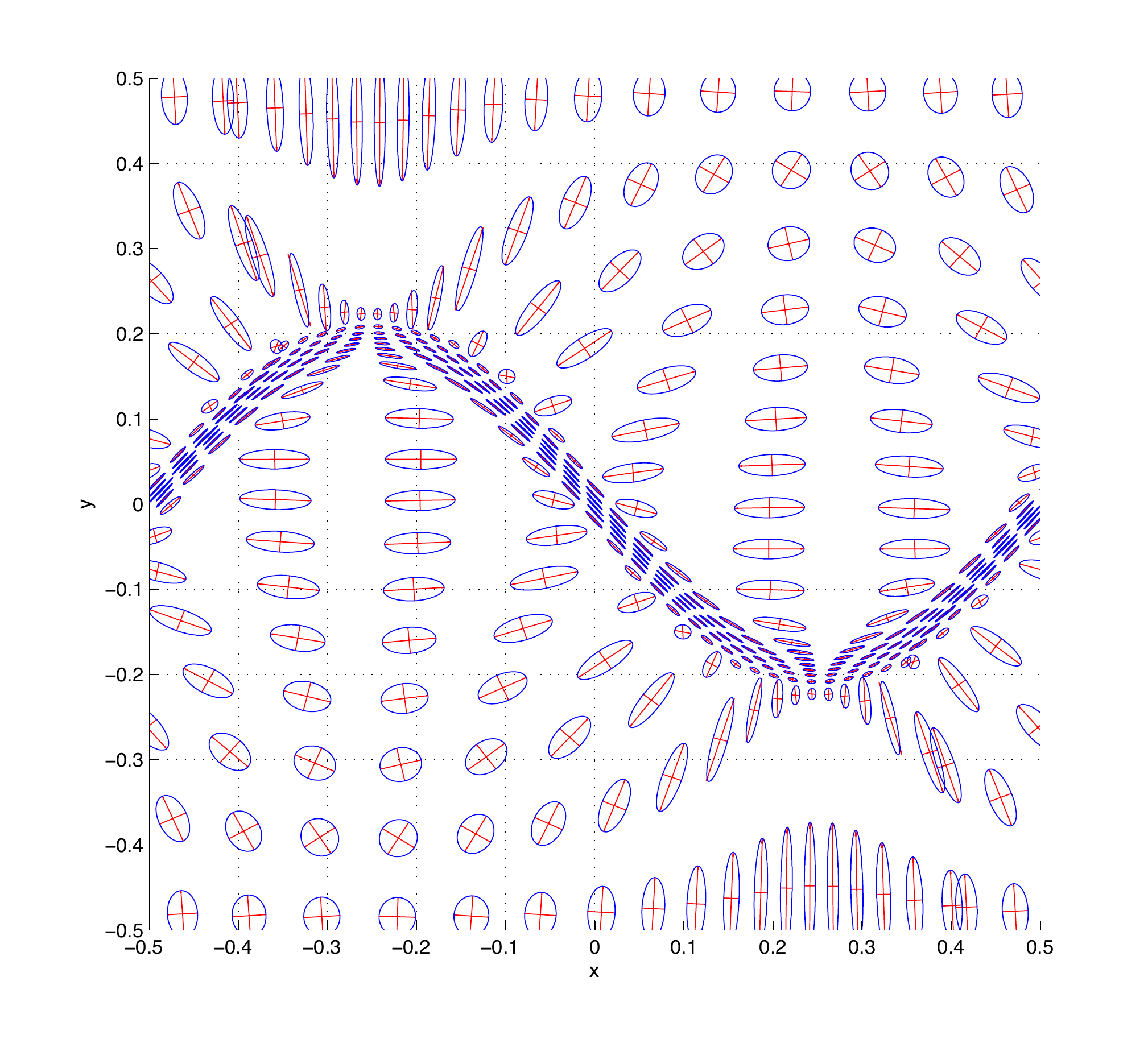} \end{center}
\caption{ \small The computed mesh for the density function (\ref{sin}) with $\alpha_1=50$, $\alpha_2=50$, $\theta=3$, and $a=0.25$ with Neumann boundary conditions in the y-direction and periodic in the x-direction (left), and  an eigenplot of $\mathbf{J}$ (right). }
\label{sinemesh}
\end{figure}
We note that when $\theta$ is greater than $1$ the approximation underestimates the level of a skewness close to the top and bottom boundary. Furthermore, due to the Neumann boundary condition the eigenvectors are not aligned tangential and orthogonal to $\Psi$ at the boundary. 

\begin{figure}[hhhhhhhh!!!!!]
\begin{center}\includegraphics[height=6cm,width=6.5cm]{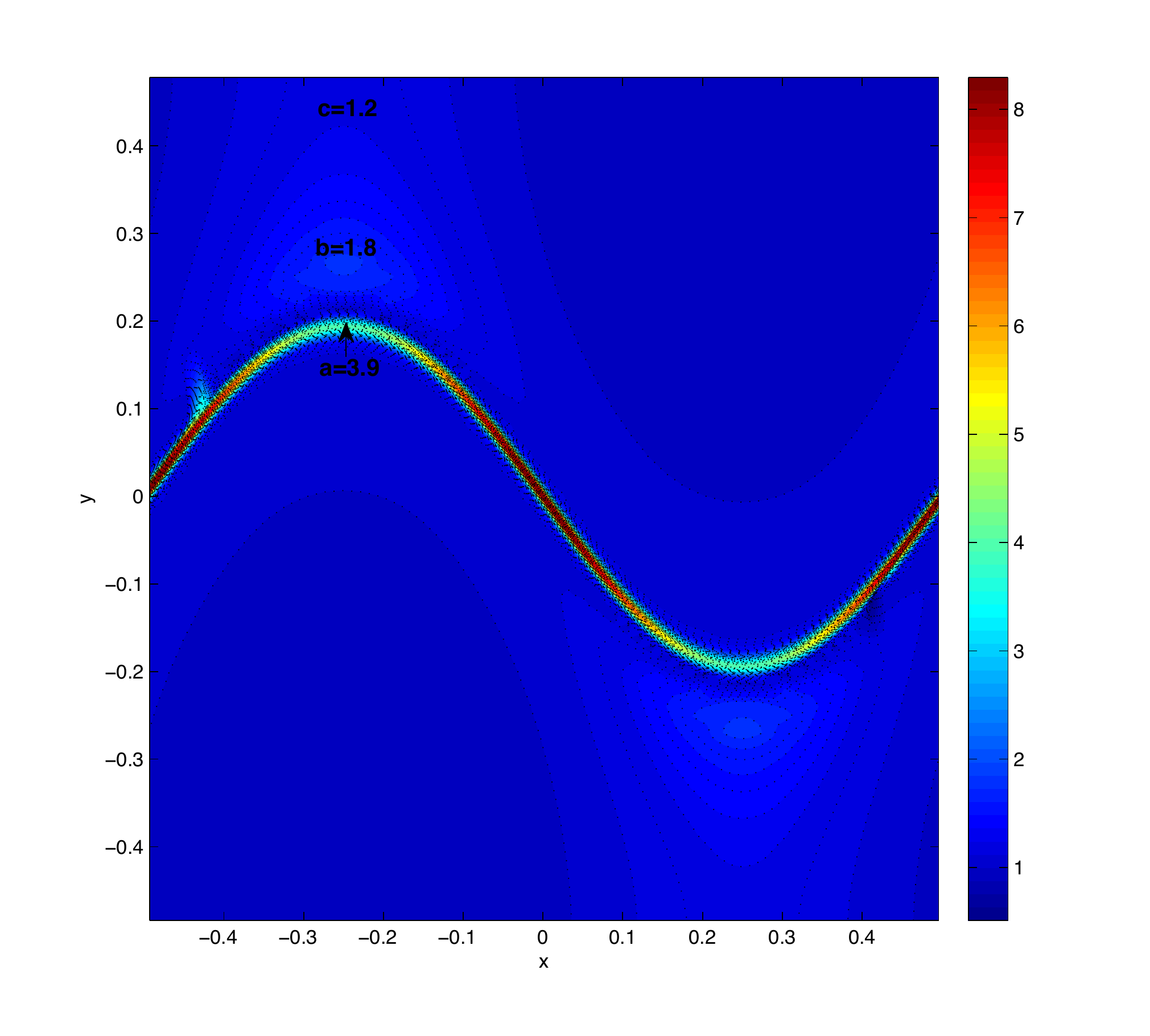}\includegraphics[height=6cm,width=6.5cm]{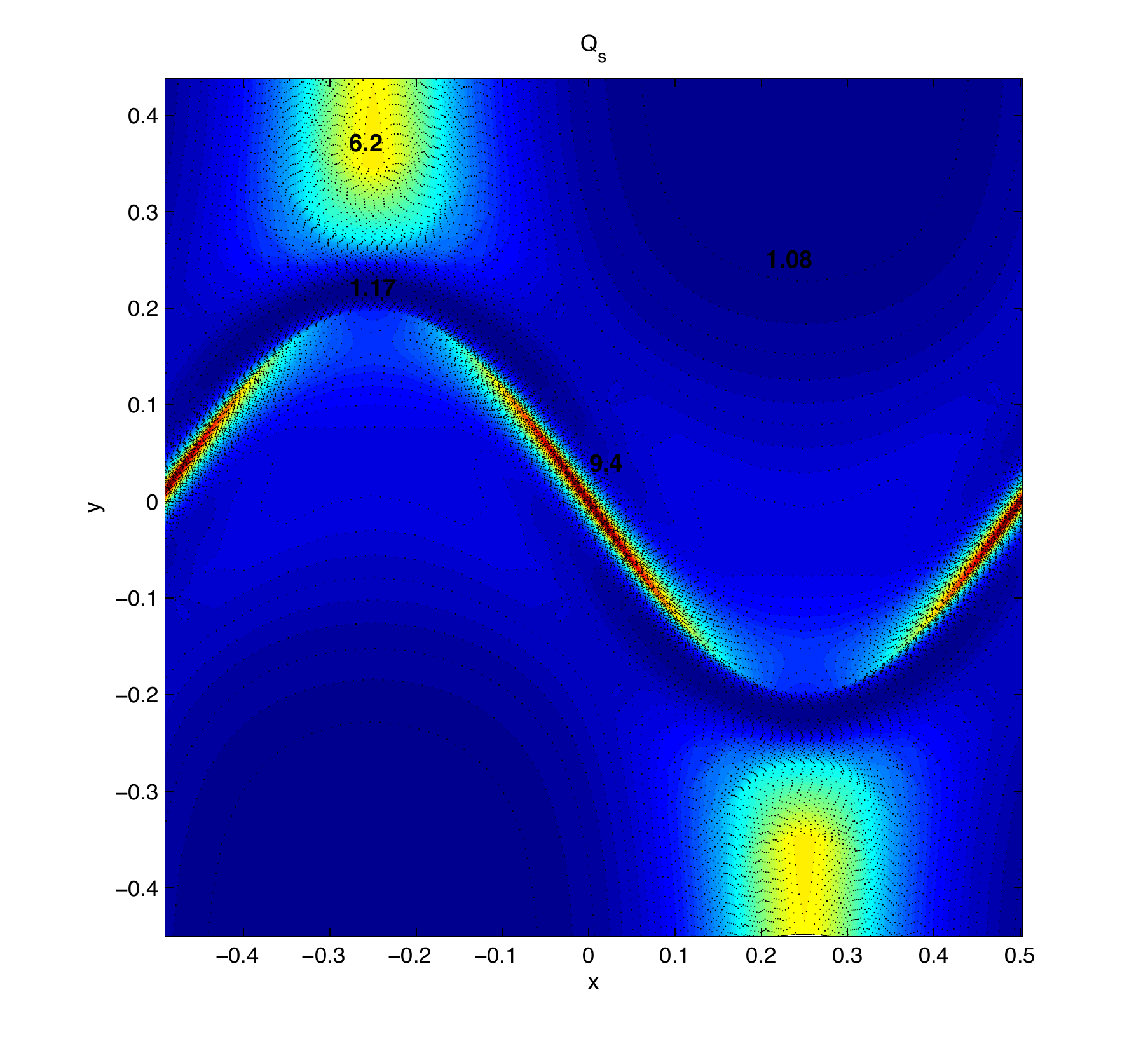}\end{center}
\caption{ \small The value of $\hat{Q}_s$ for the numerically computed mesh with density function (\ref{sin}) and Neumann boundary conditions in the y-direction and periodic in the x-direction when $\alpha_1=20$, $\alpha_2=100$, $\theta=1.4$, and $a=0.25$ (left), $\alpha_1=50$, $\alpha_2=50$, $\theta=3$ (right) . }
\label{qs_sin}
\end{figure}

\subsection{Example 2: Time Dependant Solution of a nonlinear PDE}
We now consider the adaptive numerical solution of the Buckley Leverett equation
\begin{equation}
{
u_t+f(u)_x+g(u)_y=\dot{x}u_x+\dot{y}u_y+\mu \triangle u,
}
\label{BL1}
\end{equation}
with $\mu=1.1\times10^{-3}$. The flux functions are
$$f(u) = \frac{u^2}{3(u^2+(1-u)^2)},\quad g(u)=\frac{1}{3}f(u)(1-5(1-u)^2),$$
and the initial data is
\[
u(x,y,0)=\left\{\begin{array}{ll}1,&(x-0.5)^2+(y-0.5)^2<\frac{1}{18}\\0&\mbox{otherwise.}\end{array}\right.
\]
This model includes gravitational effects in the y-direction. The exact solution is unknown, 
although numerical results \cite{Karlsen}, \cite{ZhangTang}, indicate a thin and curved reaction front forms which is our main motivation for studying it here.
The solution to (\ref{BL1}) is computed on the domain $[0,1]^2$ up to time t=0.4. To compute this solution the mesh is continuously updated by solving a parabolised version (PMA) of the MA equation as described in \cite{BW:09}.
The coupled system of the Buckley Leverett equation  and PMA is then solved in the computational domain using an alternate procedure with a composite centred finite difference scheme used to discretise both systems. For this calculation we use an arc-length based density function given by
\[
\rho =\sqrt{1+\vert \nabla u\vert^2}.
\]
The solution and mesh at $t=0.4$ are shown in Fig.~\ref{blmesh}.
\begin{figure}[hhhhhhhh!!!!!]
\begin{center}\includegraphics[height=6cm,width=6.5cm]{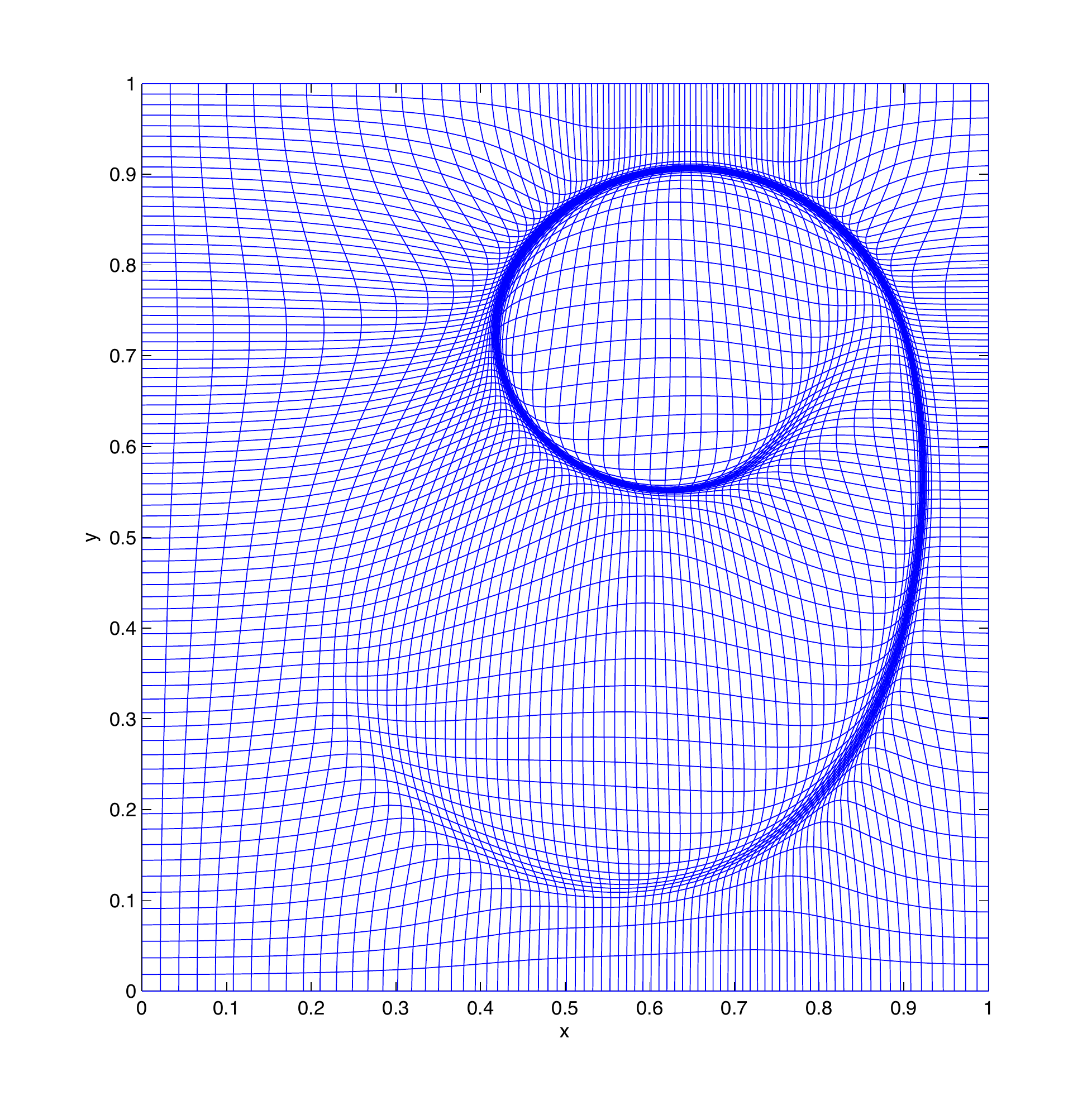}\includegraphics[height=6cm,width=6.5cm]{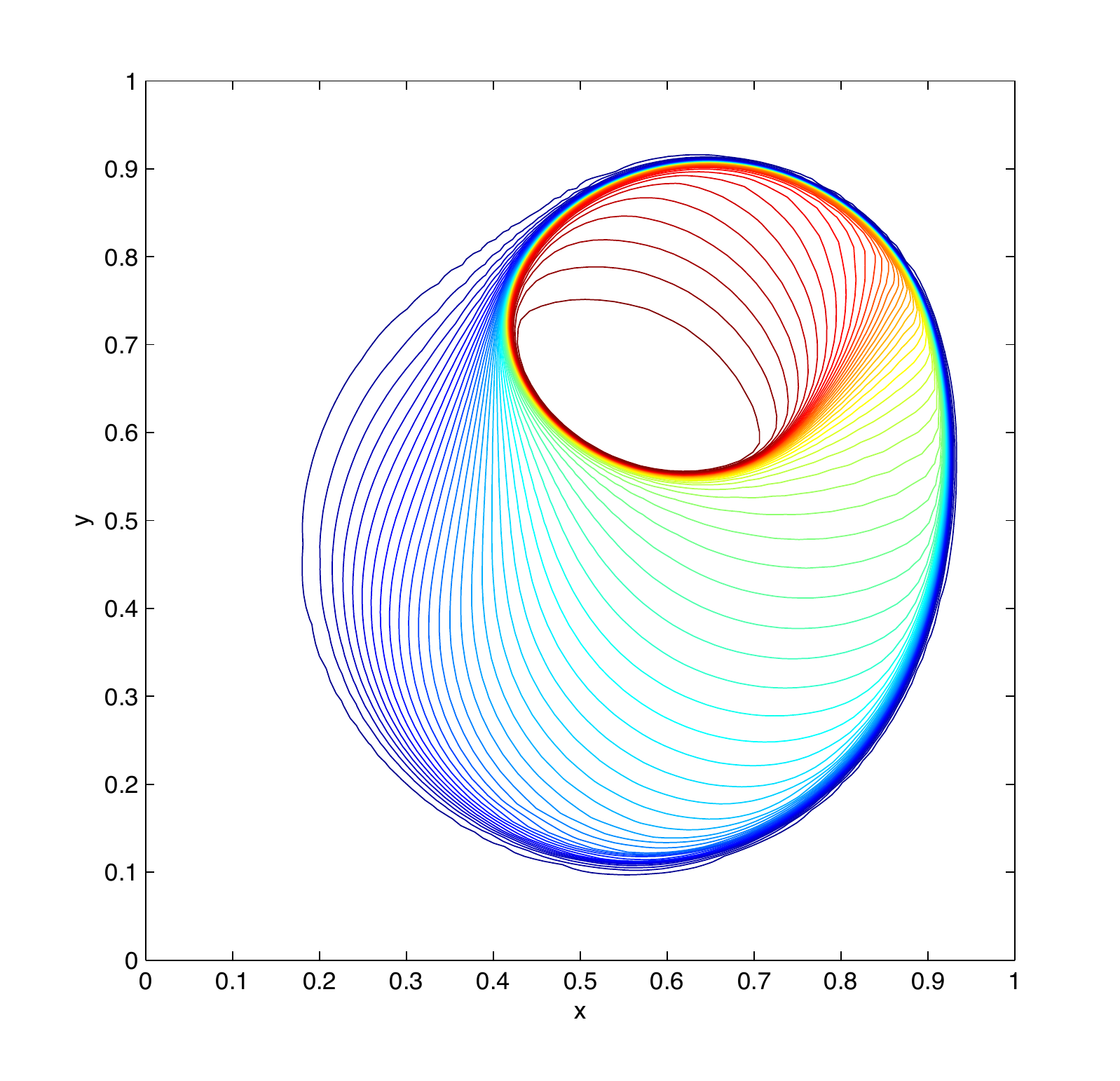}\end{center}
\caption{ \small The numerically computed mesh $(80\times80)$ with Neumann boundary conditions for the Buckley Leverett problem at t=0.4 (left), the solution (right). }
\label{blmesh}
\end{figure}
In Fig.~\ref{blmesh2} a plot of the circumscribed elipses of the Jacobian for a number of mesh elements reveals that in the region where the density function is large the eigenvectors are tangential and orthogonal to the feature. Furthermore the eigenvectors remain aligned to this feature as the solution and mesh evolves in time. 
\begin{figure}[hhhhhhhh!!!!!]
\begin{center}\includegraphics[height=6cm,width=6.5cm]{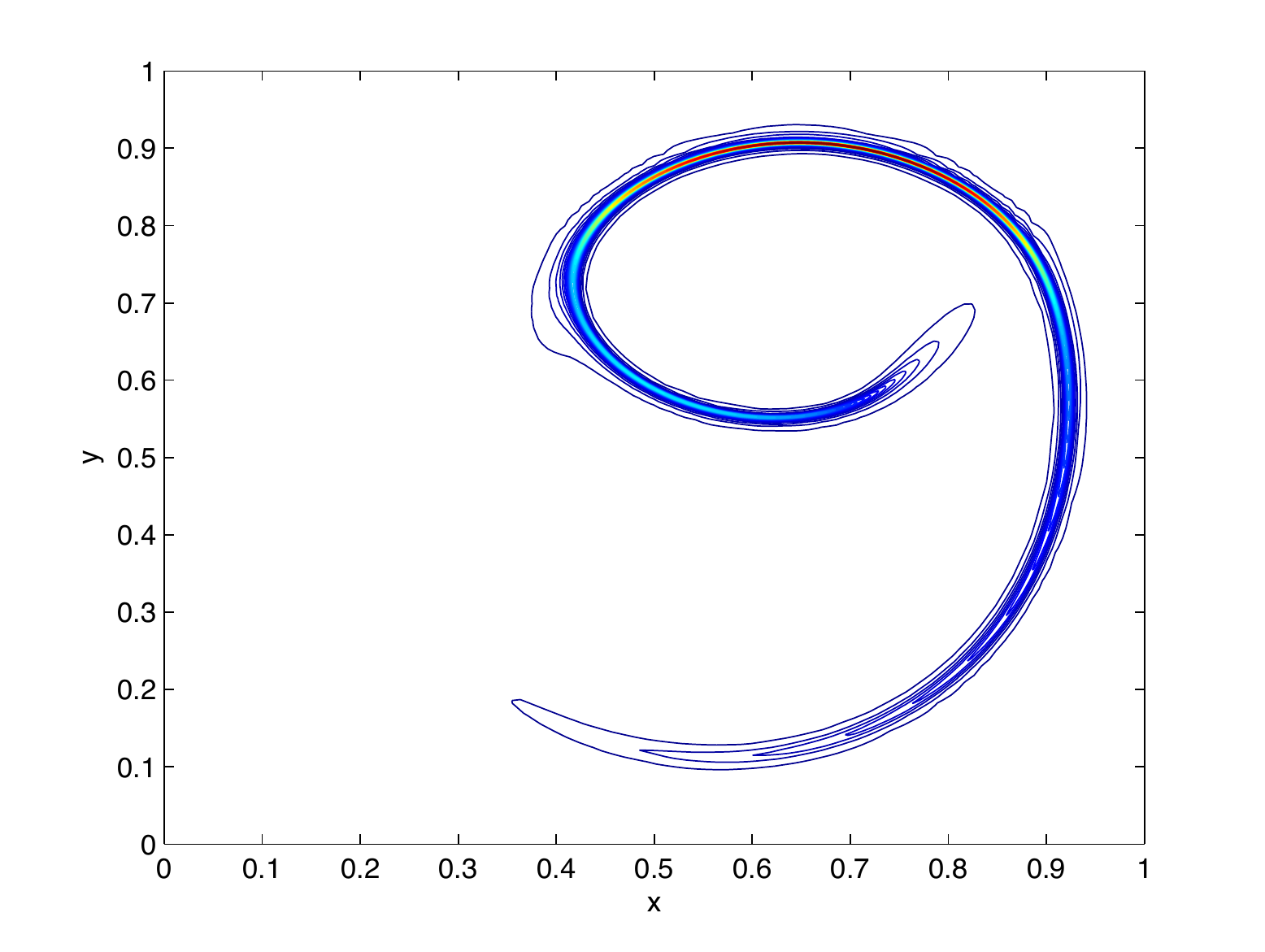}\includegraphics[height=6cm,width=6.5cm]{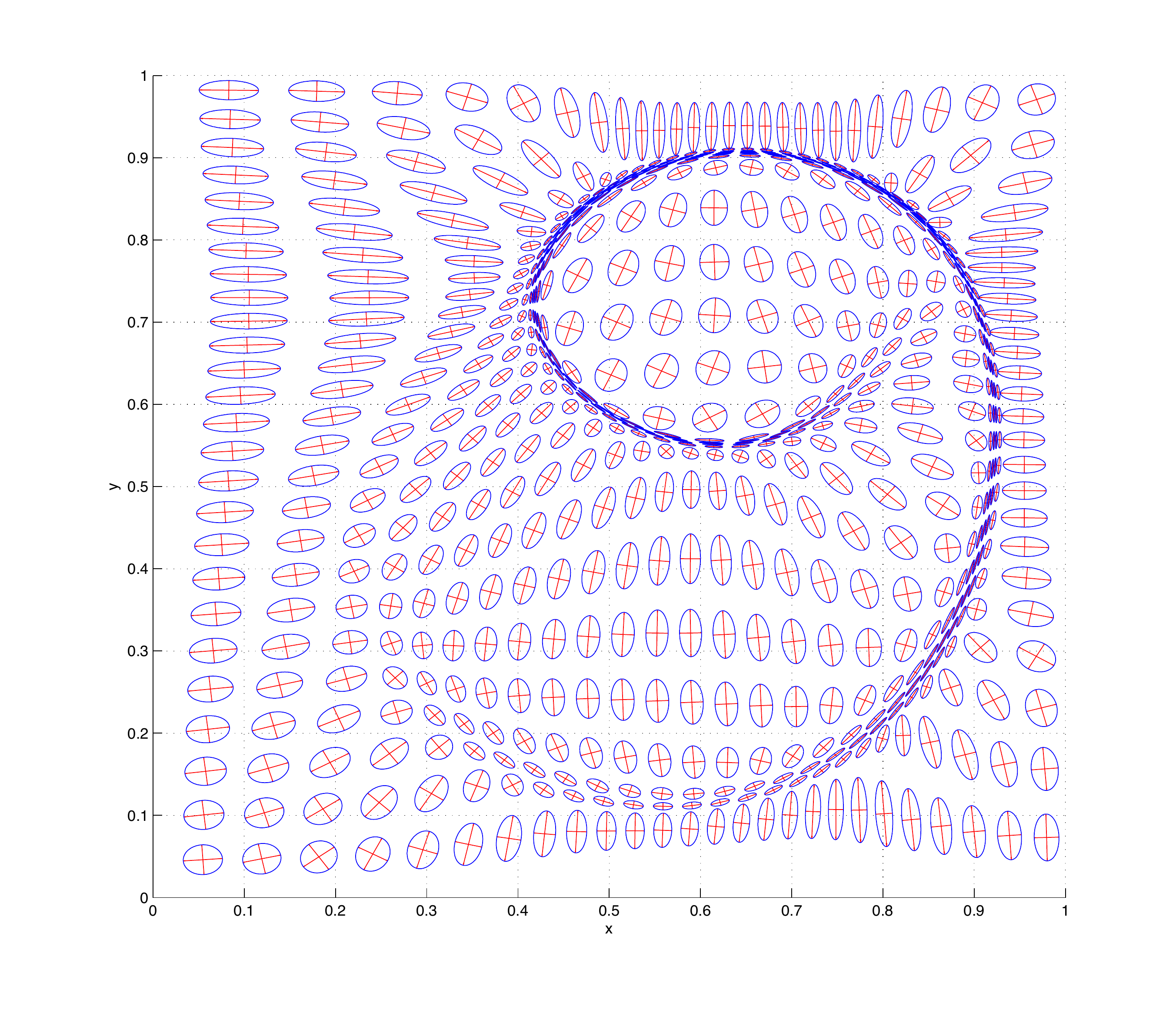}\end{center}
\caption{ \small The density function $\rho$ at $t=0.44$ for the Buckley Leverett problem (left), and  circumscribed ellipses of the Jacobian for the corresponding mesh (right). }
\label{blmesh2}
\end{figure}
A comparison of the eigenvalues with those associated with a linear feature shows that this is an excellent approximation along regions of the curve that are close to linear  (see Fig.~\ref{BL_lin_zoom} (left)).
Moreover, in regions of high curvature a radially symmetric solution gives a much better approximation. The density function in a region of high curvature is considered to be part of a radially symmetric feature with density function $\tilde{\rho}$ 
\[
\tilde{\rho}=1+\alpha_1 \mathrm{sech}^2(\alpha_2\vert \Psi_2\vert)\nonumber, \hspace{1cm}\Psi_2=\tilde{R}^2-{a}^2,\nonumber
\]
where $\tilde{R}=\sqrt{(x-0.62)^2+(y-.72)^2}$, $a=0.2$, $\alpha_1=70$, and $\alpha_2=500$.
A comparison of the eigenvalues of the Jacobian with those associated with the density function $\tilde{\rho}$, which are computed using the radially symmetric solution, are shown in a region of high curvature in Fig. \ref{BL_lin_zoom} (right).
\begin{figure}[hhhhhhhh!!!!!]
\begin{center}\includegraphics[height=6cm,width=6.5cm]{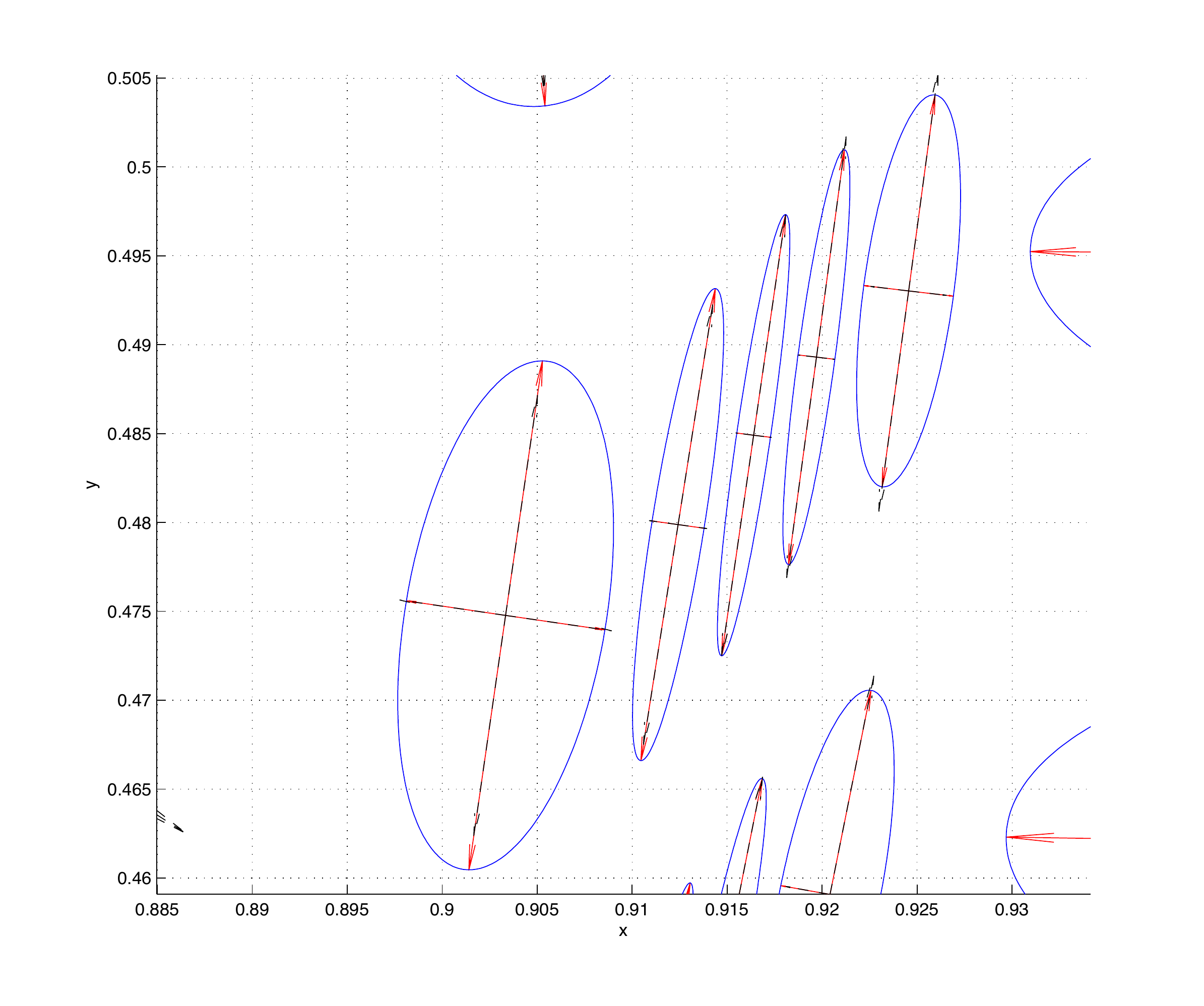}\includegraphics[height=5.8cm,width=6cm]{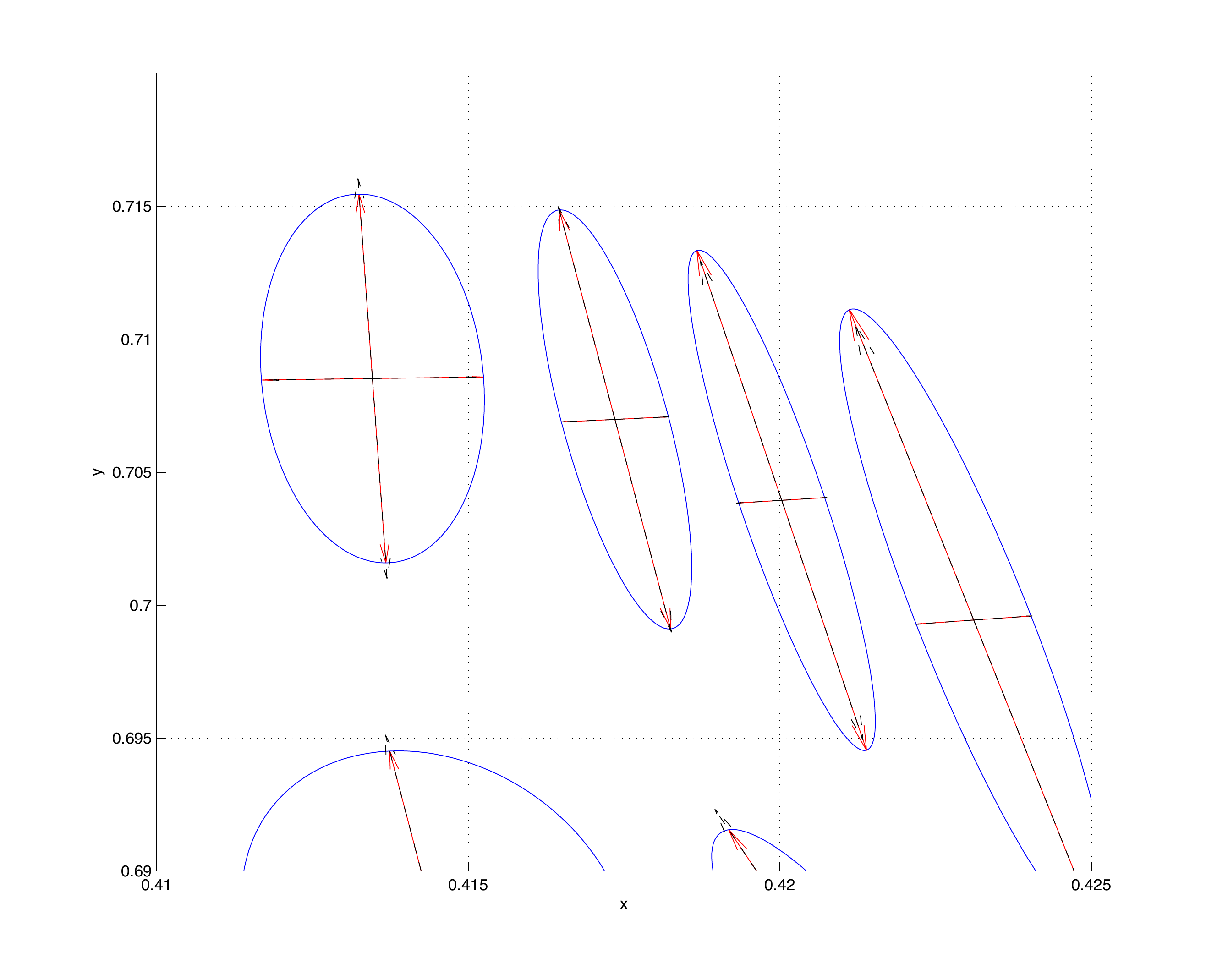}\end{center}
\caption{ \small A comparison of the eigenvalues of the Jacobian (red) with those corresponding to the linear solution (black) in a region of low curvature (left), and those corresponding to the radially symmetric solution in a region of high curvature (right). }
\label{BL_lin_zoom}
\end{figure}

\section{Conclusions}
We have shown that a  mesh redistribution method that is based on equidistributing a scalar density function via solving the Monge-Amp\`ere equation has the
capability of producing naturally anisotropic meshes in regions of rapid change in the solution structure. Furthermore, we have rigorously shown this for model
problems comprising orthogonal linear features and radially symmetric features by
deriving the exact metric tensor to which these meshes align. We have also demonstrated that the  
results for these linear and radially symmetric cases can be used to approximate alignment for more complicated flow structures that arise in the solution of a non-linear PDE.  
The metric tensor has a very similar form to those traditionally used in variational methods, and given that determination of such a tensor
can be a difficult task, it would definitely be advantageous if an optimal metric tensor arose naturally from the solution of the Monge-Amp\`ere equation. This is indeed fascinating, and a closer examination of how this metric tensor is related to those known to minimise interpolation error is the subject of ongoing research.

\section*{Acknowledgments}The authors thank Phil Browne (University of Reading), Mike Cullen (UK Met Office), Weizhang Huang (University of Kansas), and  J F Williams (Simon Fraser University), for many useful suggestions and encouragement. The research of the second and third author is partially supported by NSERC Grant A8781.

\end{document}